\documentclass[a4paper,10pt]{article}%
\usepackage{amsmath}%
\usepackage{amsfonts}%
\usepackage{amssymb}%
\usepackage{amsthm}
\usepackage{graphicx}
\usepackage[latin1]{inputenc}
\usepackage[T1]{fontenc}
\usepackage[english,french]{babel}

\newtheorem*{theo}{Theorem}

\newtheorem{theorem}{Theorem}[section]

\newtheorem{claim}[theorem]{Claim}

\newtheorem{corollary}[theorem]{Corollary}

\newtheorem{definition}[theorem]{Definition}

\newtheorem{lemma}[theorem]{Lemma}

\newtheorem{proposition}[theorem]{Proposition}
\newtheorem{remark}[theorem]{Remark}

\newlength{\espaceavantspecialthm}
\newlength{\espaceapresspecialthm}
{\\}

\newenvironment{defi*}[1][]{
\vskip \espaceavantspecialthm \noindent \textbf{D\'efinition.} }%
{\vskip \espaceapresspecialthm}

\begin{document}

\sloppy

\title{Groups of smooth diffeomorphisms of Cantor sets embedded in a line}
\author{Dominique Malicet, Emmanuel Militon}
\maketitle

\setlength{\parskip}{10pt}

\selectlanguage{english}

\begin{abstract}
Let $K$ be a Cantor set embedded in the real line $\mathbb{R}$. Following Funar and Neretin, we define the diffeomorphism group of $K$ as the group of homeomorphisms of $K$ which locally look like a diffeomorphism between two intervals of $\mathbb{R}$. Higman-Thompson's groups $V_{n}$ appear as subgroups of such groups. In this article, we prove some properties of this group. First, we study the Burnside problem in this group and we prove that any finitely generated subgroup consisting of finite order elements is finite. This property was already proved by Rover in the case of the groups $V_{n}$. We also prove that any finitely generated subgroup $H$ without free subsemigroup on two generators is virtually abelian. The corresponding result for the groups $V_{n}$ was unknown to our knowledge. As a consequence, those groups do not contain nilpotent groups which are not virtually abelian.  
\end{abstract}

\section{Introduction}

We call \emph{Cantor set} any compact totally disconnected set $K$ such that any point of $K$ is an accumulation point.

When studying the dynamics of an action of a group $G$ on a closed surface, it is convenient to look at minimal subsets of the action. Recall that a $G$-invariant closed nonempty subset $K$ of our surface is a minimal subset for the action of $G$ on our surface if every orbit of points in $K$ is dense in $K$. This is equivalent to saying that $K$ is minimal for the inclusion relation among $G$-invariant closed nonempty subsets. Zorn's lemma ensures that such subsets always exist. A typical case which can occur is the case where this minimal subset turns out to be a Cantor set $K$.

In this article, we will restrict to the case where our Cantor set is embedded in a line, \emph{i.e.} embedded in a one-dimensional submanifold diffeomorphic to $\mathbb{R}$.

We will give two equivalent definitions of the group we are interested in. 

\begin{definition} \label{extrinsic}
Let $r$ be an integer greater than or equal to $1$ or $+ \infty$. Let $K$ be a Cantor set contained in a line $L$ which is $C^{r}$-embedded in a manifold $M$ with $\mathrm{dim}(M) \geq 2$. The group of $C^{r}$-diffeomorphisms of $K$, denoted $\mathfrak{diff}^{r}(K)$, is the group of restrictions to $K$ of $C^{r}$-diffeomorphisms $f$ of $M$ such that $f(K)=K$.
\end{definition}

\begin{remark}
The isomorphism class of this group is independent of the embedding of the line $L$ in $M$ and of the manifold $M$, as long as $\mathrm{dim}(M) \geq 2$. This is a consequence of the second definition below, which is independent of $L$ and $M$, and of the equivalence between the two definitions. However, if we look at the same group in the case where $M$ is a circle, we only obtain a strict subgroup of the latter group as elements of the group have to preserve a cyclic order.
\end{remark}

\begin{remark} If $G$ is a group acting on a manifold $M$ by $C^{r}$-diffeomorphisms with such a Cantor set $K$ as a minimal set, then there exists a nontrivial morphism $G \rightarrow \mathfrak{diff}^{r}(K)$. If we understand well the group $\mathfrak{diff}^{r}(K)$, we can obtain information on which group can act on $M$ with such a minimal invariant set.
\end{remark}

We now give a second definition of our group: the group $\mathfrak{diff}^{r}(K)$ is the group of homeomorphisms of $K$ which locally coincide with a $C^{r}$-diffeomorphism of an open interval of $\mathbb{R}$ (see precise definition below). We will prove the equivalence between the two definitions in Section 2 (see Proposition \ref{equivalence} for a precise statement).

\begin{definition} \label{intrinsic}
Let $K$ be a Cantor set contained in $\mathbb{R}$.  The group $\mathfrak{diff}^{r}(K)$ is the group of homeomorphisms $f$ of $K$ such that, for any point $x$ in $K$, there exists an open interval $I$ of $\mathbb{R}$ containing $x$ and a $C^{r}$-diffeomorphism $\tilde{f}: I \rightarrow \tilde{f}(I)$ such that $\tilde{f}_{|I \cap K} =f_{|I \cap K}$.
\end{definition}

\begin{remark} we can adapt this second definition to other kinds of regularity. For instance, in this article, we will denote by $\mathfrak{diff}^{1+Lip}(K)$ the group of homeomorphisms of $K$ which locally coincide with a $C^{1+Lip}$-diffeomorphism between two intervals of $\mathbb{R}$, that is a $C^{1}$-diffeomorphism $\tilde{f}$ such that $\log(\tilde{f}')$ is Lipschitz continuous.
\end{remark}

The generalizations of those two definitions to the case $r=0$ are not equivalent, but we will use this second definition to define the group $\mathfrak{diff}^{0}(K)$ : it is the group of homeomorphisms of $K$ which coincide locally with a homeomorphism between two intervals of $\mathbb{R}$.

In the article \cite{FN}, Funar and Neretin have computed these groups in many cases and provided examples of Cantor subsets for which these groups are trivial. They have in particular computed this group in the case where $K$ is the standard ternary Cantor subset, which we call $K_{2}$.  Let us recall first the construction of $K_{2}$. Start with the segment $[0,1]$. Cut this interval into three equal pieces $[0,\frac{1}{3}]$, $[\frac{1}{3},\frac{2}{3}]$ and $[\frac{2}{3}, 1]$. Now, throw out the middle segment: we obtain a new compact set $[0,\frac{1}{3}] \cup  [\frac{2}{3}, 1]$. Now remove the middle third of each of these intervals : we obtain the compact subset $[0,\frac{1}{9}] \cup [\frac{2}{9},\frac{1}{3}] \cup [\frac{2}{3},\frac{7}{9}] \cup [\frac{8}{9},1]$. Then repeat the procedure for each of the obtained intervals. We obtain a decreasing sequence of compact subsets : the intersection of this sequence is the subset $K_{2}$ (see Figure \ref{constructionCantor}). More generally, if we remove $n-1$ evenly spaced intervals at each step instead of one, we obtain a Cantor set which we denote by $K_{n}$.

\begin{figure}[ht] 
\begin{center}
\includegraphics[scale=0.5]{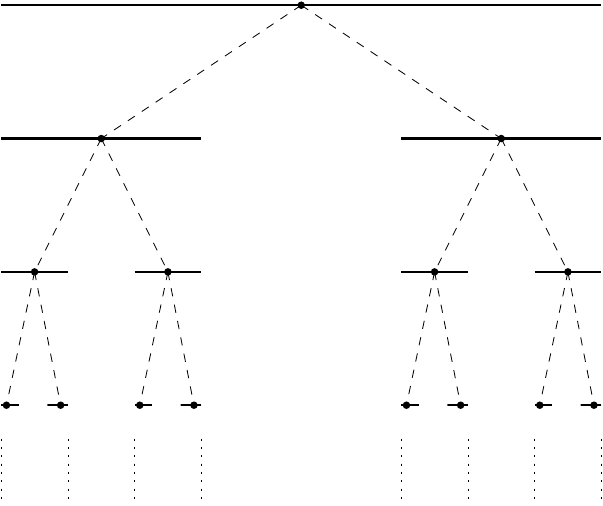}
\caption{The first steps of construction of $K_{2}$ and bijection of the intervals with the vertices of a binary tree}
\label{constructionCantor}
\end{center}
\end{figure}

For convenience, we will call \emph{interval of} $K_{2}$ the intersection of $K_{2}$ with one of the intervals appearing in this construction (this corresponds to cylinders in symbolic dynamics). The set of intervals of $K_{2}$ is a basis of the topology of $K_{2}$. This basis consists of clopen subsets, \emph{i.e.} subsets which are closed and open. The set of intervals of $K_{2}$ is in bijective correspondence with the set of vertices of a rooted binary tree (see Figure \ref{constructionCantor}).

We now give a procedure to construct elements of $\mathfrak{diff}^{\infty}(K_{2})$ (see Figure \ref{diffeoK2} for an example of a diffeomorphism of $K_{2}$).

\underline{Step 1:} Choose two finite partitions of $K_{2}$ by intervals of $K_{2}$ which have the same cardinality. We denote these partitions by $\left\{ I_{i}, \ 1 \leq i \leq m \right\}$ and $\left\{ J_{j}, \ 1 \leq j \leq m \right\}$.\\

\underline{Step 2:} Take a permutation $\sigma \in \mathcal{S}_m$. This enables to define an element $f$ of the group $\mathfrak{diff}^{\infty}(K_{2})$ in the following way. The restriction of $f$ to $I_{i}$ is the unique orientation preserving affine map which sends the interval $I_{i}$ onto the interval $J_{j}$. Such a map sends $I_{i} \cap K_{2}$ to $J_{j} \cap K_{2}$.

\underline{Step 3:} Chose a subset $A$ of the second partition $\left\{ J_{j}, \ 1 \leq j \leq m \right\}$ and flip each of the intervals in $A$, \emph{i.e.} compose the diffeomorphism obtained in Step 2 with the diffeomorphism whose restriction to $J_{j}$ is the identity if $J_{j} \notin A$ and is the symmetry with respect to the midpoint of $J_{j}$ if $J_{j} \in A$.

\begin{figure}[h] 

\begin{center}
\includegraphics[scale=0.5]{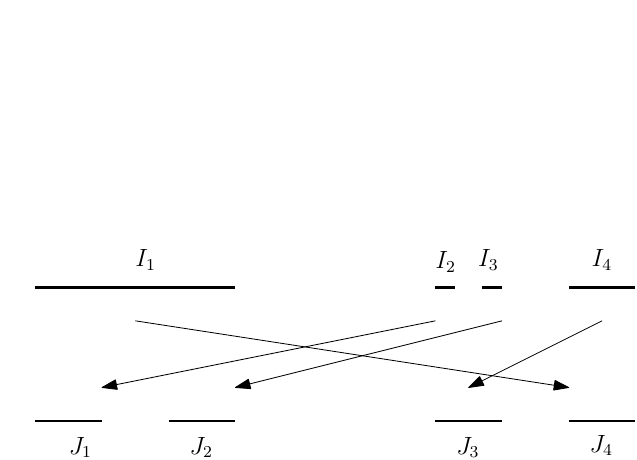}
\caption{An example of diffeomorphism of $K_{2}$}
\label{diffeoK2}
\end{center}

\end{figure}

Of course, we can define a similar algorithm to construct elements of $\mathfrak{diff}^{r}(K_{n})$.

\begin{theo}[Funar-Neretin]
For any $r \geq 1$ or for $r=\infty$, the group $\mathfrak{diff}^{r}(K_{n})$ is the group consisting of elements constructed following the above procedure.
\end{theo}
 
Strictly speaking, Funar and Neretin proved only the case $n=2$ but their proof also applies in the case of the Cantor sets $K_{n}$ for $n>2$. Notice that this group does not depend on the regularity $r \geq 1$ : this seems to be a consequence of the ``regular'' shape of this Cantor set.

Inside $\mathfrak{diff}^{r}(K_{n})$, there is a natural subgroup : the subgroup consisting of elements which are constructed using only the first two steps of the above procedure. This subgroup is the well-known Higman-Thompson group $V_{n}$ (see the introduction of the article \cite{BBG} for more background on Higman-Thompson groups).

In this article, we prove some general results about the groups $\mathrm{diff}^{r}(K)$. These theorems suggest that those groups share common features with rank one simple Lie groups. In what follows, we fix a Cantor set $K$ embedded in $\mathbb{R}$.

\subsection{Burnside property}

\begin{definition} A group is \emph{periodic} if any of its elements has finite order.
\end{definition}

In 1902, Burnside asked whether there exist finitely generated periodic groups which are infinite (see \cite{Bur}). Nine years later, Schur managed to prove that any finitely generated periodic group which is a subgroup of $GL_{n}(\mathbb{C})$ has to be finite. Much later, in the 60's, Golod and Shafarevich proved in \cite{GoS} that there exist infinite finitely generated periodic groups. Many more examples were constructed later.

\begin{theorem} \label{Burn}
Any periodic finitely generated subgroup of $\mathfrak{diff}^{1+Lip}(K)$ is finite.
\end{theorem}

We are not able to lower the regularity to $C^{1}$ in this theorem for the moment. However, observe that the same theorem is false in the case of the group of homeomorphisms of a Cantor set as any finitely generated group is a subgroup of this group. To see this, observe that any infinite countable group $G$ acts continuously and faithfully on $\left\{ 0,1 \right\}^{G}$, which is a Cantor set for the product topology and recall that any two Cantor sets are homeomorphic.

We prove Theorem \ref{Burn} in Section 3 of this article.

As a consequence, as the Higman-Thompson groups $V_{n}$ are subgroups of groups of $C^{\infty}$-diffeomorphisms of Cantor sets, we have a new proof of the following theorem by Rover (see \cite{Rov}).

\begin{theo}[Rover]
Let $n \geq 2$. Any finitely generated periodic subgroup of $V_{n}$ is finite.
\end{theo}

The proof of Theorem \ref{Burn} relies on an adaptation of standard $1$-dimensional tools from dynamical systems, namely Sacksteder's theorem and the Thurston stability theorem.

\subsection{Subgroups without free subsemigroups on two generators}

In this section, we will consider finitely generated subgroups of the group $\mathfrak{diff}^{1+Lip}(K)$ without free subsemigroups on two generators. This includes all finitely generated groups of subexponential growth. In particular, this includes
nilpotent finitely generated subgroups, which we define below. 

Fix a group $G$. If $H$ and $H'$ are subgroups of $G$, we define $[H,H']$ as the subgroup of $G$ generated by elements of the form $[h,h']=hh'h^{-1}h'^{-1}$, where $h \in H$ and $h' \in H'$.

We define a sequence $(G_{n})_{n \geq 1}$ of subgroups of $G$ by the following relations:
$$G_{1}=G$$
$$ \forall n \geq 1, G_{n+1}=[G_{n},G].$$

A group $G$ is said to be \emph{nilpotent} if there exists $n \geq 1$ such that $G_{n}=\left\{ 1 \right\}$. If the group $G$ is nilpotent and nontrivial, its \emph{degree} is the smallest integer $n$ such that $G_{n}$ is nontrivial and $G_{n+1}$ is trivial. A typical example of finitely generated nilpotent group is the Heisenberg group $\mathcal{H}$ with integer coefficient. This group $\mathcal{H}$ is the group of upper unitriangular $3 \times 3$ matrices with integer coefficients.

The following theorem states that subgroups of $\mathfrak{diff}^{1+Lip}(K)$ without free subsemigroups on two generators are close to being abelian.

\begin{theorem} \label{nosubsemi}
Let $G$ be a finitely generated subgroup of $\mathfrak{diff}^{1+Lip}(K)$ without free subsemigroups on two generators. Then the group $G$ is virtually abelian.
\end{theorem}

Recall that, by definition, a group is virtually abelian if it contains a finite-index subgroup which is abelian. We cannot hope for a better conclusion in Theorem \ref{nosubsemi}. Indeed, take $K=K_{2}$ the standard ternary Cantor set. Let $S$ be the group of diffeomorphisms of $K_{2}$ which permute the intervals $[0,\frac{1}{9}]$, $[\frac{2}{9},\frac{1}{3}]$, $[\frac{2}{3},\frac{7}{9}]$, $[\frac{8}{9},1]$  and whose restriction to each of these intervals is orientation-preserving and affine. The group $S$ is isomorphic to the finite group $\mathcal{S}_{4}$, the symmetric group on $4$ elements. Take also an infinite order element $f$ of $\mathfrak{diff}^{\infty}(K_{2})$ which is supported in $[0,\frac{1}{9}]\cap K_{2}$, meaning that it pointwise fixes the points outside $[0,\frac{1}{9}]\cap K_{2}$. Then the subgroup of $\mathfrak{diff}^{\infty}(K_{2})$ generated by $f$ and $S$ is virtually abelian : it contains the group $\mathbb{Z}^4$ as a finite index subgroup. But it is not abelian. Notice that, with this kind of construction, we can obtain any virtually abelian group as a subgroup of a group of diffeomorphisms of a Cantor set. 

\begin{remark} Observe that finitely generated groups of subexponential growth contain no free subsemigroup on two generators. Hence subgroups of $\mathfrak{diff}^{1+Lip}(K)$ of subexponential growth are virtually abelian. \end{remark}

\begin{remark} Notice that Theorem \ref{nosubsemi} implies Theorem \ref{Burn} as periodic groups contain no free subsemigroup on two generators and finitely generated abelian periodic groups are known to be finite. However, we use Theorem \ref{Burn} (and even a stronger version of it which is Proposition \ref{Burngen}) to prove Theorem \ref{nosubsemi}. That is why we will first prove Theorem \ref{Burn} in this article.
\end{remark}

\begin{remark} In Theorem \ref{nosubsemi}, the regularity cannot be lowered to a $C^1$ regularity. Indeed, Farb and Franks proved in \cite{FF} that any torsion-free finitely generated nilpotent group is a subgroup of the group of homeomorphisms of the compact interval $[0,1]$. For each nilpotent subgroup that they construct, it is possible to find a Cantor subset which is invariant. Hence any torsion-free finitely generated nilpotent group is a subgroup of the group of $C^{1}$ diffeomorphisms of some Cantor set.\end{remark}

If we only assumed that our group contained no free subgroups, we could not prove that our group is virtually abelian. Indeed, the Thompson group $F$ is contained in $V_{2}$ and is hence a subgroup of the group of diffeomorphisms of the standard ternary Cantor set. But this group is finitely generated, contains no free subgroup on two generators and is not abelian. For more information about the group $F$ and references for proofs of these results, see Section 1.5 of \cite{Nav}. The best we can hope for finitely generated subgroups without a free subgroup on two generators is that they have a finite orbit. We will explore this question in a forthcoming article.

We prove Theorem \ref{nosubsemi} in Section 4 of this article.

Observe that the corresponding statement for Higman-Thompson's groups $V_{n}$ was unknown, as far as we know. We state it as a corollary.

\begin{corollary}
Let $n \geq 2$. Any subgroup of the group $V_{n}$ without free subsemigroup on two generators is virtually abelian.
\end{corollary}

As the derived subgroup of a finitely generated nilpotent group is finitely generated and has no free subsemigroups on two generators, we obtain the following statement.

\begin{corollary} \label{nilpotent}
Any finitely generated nilpotent subgroup of $\mathfrak{diff}^{1+Lip}(K)$ is virtually abelian.
\end{corollary}

In particular, there is no Heisenberg group with integer coefficients as subgroup of $\mathfrak{diff}^{r}(K)$ for $r \geq 2$. 

Higman-Thompson groups $V_{n}$ were already known to satisfy this corollary by a result by Bleak, Bowman, Gordon Lynch, Graham, Hughes, Matucci and Sapir \cite{BBG} about distorted cyclic subgroups.

As a consequence of this corollary, we obtain the following statement, which is related to the Zimmer conjecture.

\begin{theorem}
Let $r \geq 2$ and $\Gamma$ be a finite index subgroup of $SL_{n}(\mathbb{Z})$ (or any almost simple group which contains a nonabelian nilpotent group whose derived subgroup is infinite). Any morphism $\Gamma \rightarrow \mathfrak{diff}^{r}(K)$ has a finite image. 
\end{theorem}

\begin{proof}
As the group $\Gamma$ contains a nonabelian nilpotent group whose derived subgroup is infinite and by Corollary \ref{nilpotent}, any morphism $\Gamma \rightarrow \mathfrak{diff}^{r}(K)$ has an infinite kernel. But the group $\Gamma$ is almost simple, which means that any normal subgroup of $\Gamma$ is either finite or a finite index subgroup of $\Gamma$. This implies that the kernel of a morphism $\Gamma \rightarrow \mathfrak{diff}^{r}(K)$ is a finite index subgroup of $\Gamma$: the image of this morphism is finite.
\end{proof}

Navas proved in \cite{Nav2} that subgroups of $C^{1+Lip}$-diffeomorphisms of the half-line without free subsemigroups on two generators are abelian. To prove Theorem \ref{nosubsemi}, we try to adapt his techniques. However, this adaptation is not easy :  any diffeomorphism of the half-line preserves the natural order on the half-line whereas, in our case, the diffeomorphisms do not a priori preserve any order on our Cantor set. Moreover, Navas is able to lower the regularity to $C^{1+bv}$ whereas we have to stick to the $C^{1+Lip}$ regularity.

\textbf{Acknowledgement:} The second author wants to thank Isabelle Liousse for a conversation which was the origin of this article.

\section{Equivalence between the two definitions}

Let $r \geq 1$ be an integer or $r=\infty$. In this section, we prove the equivalence between Definitions \ref{extrinsic} and \ref{intrinsic} of $\mathfrak{diff}^{r}(K)$. 

Let $M$ be a differential manifold with $\mathrm{dim}(M) \geq 2$. Let $L$ be a real line which is $C^{r}$-embedded in $M$. We identify the line $L$ with the real line $\mathbb{R}$. Denote by $\mathfrak{diff}^{r}(K)_{1}$ the group of $C^{r}$-diffeomorphisms of $K$ according to Definition \ref{extrinsic}, that is the group of restrictions to $K$ of $C^{r}$-diffeomorphisms of $M$ which preserve $K$. Denote by $\mathfrak{diff}^{r}(K)_{2}$ the group of $C^{r}$-diffeomorphisms of $K$ according to Definition \ref{intrinsic}, that is the group of homeomorphism of $K \subset L=\mathbb{R}$ which locally coincide with $C^{r}$-diffeomorphisms between two intervals of $\mathbb{R}$.

We prove the following statement.

\begin{proposition} \label{equivalence}
The map
$$\begin{array}{rcl}
\mathfrak{diff}^{r}(K)_{1} & \mapsto & \mathfrak{diff}^{r}(K)_{2}
\end{array}$$
which to a homeomorphism $f$ of $K$ in $\mathfrak{diff}^{r}(K)_{1}$, associates $f$, is well-defined, as it takes values in $\mathfrak{diff}^{r}(K)_{2}$, and is onto.
\end{proposition}

This amounts to showing that, if $g$ is a $C^{r}$-diffeomorphism of $M$ which preserves $K$, then the restriction $g_{|K}$ belongs to $\mathfrak{diff}^{r}(K)_{2}$ and that, if we denote by $f$ a homeomorphism of $K$ in $\mathfrak{diff}^{r}(K)_{2}$, then there exists a $C^{r}$-diffeomorphism $g$ of $M$ such that $g_{|K}=f$.

\begin{proof}[Proof of Proposition \ref{equivalence}]
Let $g$ be a $C^{r}$-diffeomorphism of $M$ which preserves $K$. Take a chart $\varphi$ defined on an open subset $U$ of $M$ onto $\mathbb{R}^{\mathrm{dim}(M)}$ such that $K \subset U$ and $\varphi(L \cap U)= \mathbb{R} \times \left\{ 0 \right\}^{\mathrm{dim}(M)-1}$. For instance, you can take as open subset $U$ a tubular neighbourhood of a segment of $L$ which contains $K$. Finally, denote by $\Pi$ the ``projection on $L$'' $\varphi^{-1} \circ p_{1} \circ \varphi$, where $p_{1}: \mathbb{R}^{\mathrm{dim}(M)} \rightarrow \mathbb{R}$ is the projection on the first coordinate.

Let $x_{0} \in K$. We will prove that the differential of $\Pi \circ g_{|L}$ at the point $x_{0}$ does not vanish. Hence, by the inverse function theorem, the map $\Pi \circ g_{|L}$ is a $C^{r}$-diffeomorphism of an open neighbourhood $I$ of the point $x_{0}$ in $L$ onto an open neighbourhood $J$ of the point $g(x_{0})$ in $L$. Moreover, $ g(I \cap K) \subset K \subset U \cap L$ so that $\Pi \circ g_{|I \cap K}=g_{| I \cap K}$ and the map $g_{|K}$ satisfies Definition \ref{intrinsic}.

It remains to show that $d(\Pi \circ g_{|L})(x_{0}) \neq 0$. As the compact subset $K$ is a Cantor set, there exists a sequence $(y_{n})_{n \in \mathbb{N}}$ of elements of $K \setminus \left\{ x_{0} \right\} $ which converges to the point $x_{0}$. Observe that, for any $n$, the point $g(y_{n})$ belongs to $K \subset L \cap U$. Hence the partial derivative 
$$ \frac{\partial}{\partial x_{1}}( \varphi \circ g \circ \varphi^{-1})(\varphi(x_{0}))$$
belongs to $\mathbb{R} \times \left\{ 0 \right\}$ and does not vanish as $g$ is a diffeomorphism. Hence  
$$ \frac{\partial}{\partial x_{1}}( p_{1} \circ \varphi \circ g \circ \varphi^{-1})(\varphi(x_{0})) \neq 0$$
and $d(\Pi \circ g_{|L})(x_{0}) \neq 0$.

Now, let $f$ be a homeomorphism of $K$ which satisfies Definition \ref{intrinsic}. By compactness of $K$, there exists a partition $(K_{i})_{1 \leq i \leq l}$ of $K$ which satisfies the following conditions.
\begin{enumerate}
\item Each subset $K_{i}$ is a clopen subset of $K$.
\item For any index $1 \leq i \leq l$, there exists an open interval $I_{i}$ of $\mathbb{R}$ such that $K_{i}= I_{i} \cap K$.
\item The intervals $I_{i}$ are pairwise disjoint.
\item For any $1 \leq i \leq l$, there exists a diffeomorphism $\tilde{f}_{i}$ defined on $I_{i}$ onto an open interval $J_{i}$ such that $\tilde{f}_{i|K_{i}}=f_{|K_{i}}$.
\item The intervals $J_{i}$ are pairwise disjoint.
\end{enumerate}
To obtain such a partition, take first any cover $(K'_{i})_{1 \leq i \leq l'}$ by clopen subsets such that conditions 1., 2. and 4. are satisfied. Take finite partitions of the clopen sets $K'_{i} \setminus \cup_{j<i} K'_{j}$ and throw away empty sets to obtain a partition such that conditions 1., 2. and 4. are satisfied. We cannot take the sets $K'_{i} \setminus \cup_{j <i} K'_{j}$ directly as those sets might not satisfy the second condition. Finally, shrink the obtained intervals $I_{i}$ in such a way that the two remaining conditions hold.

Take closed intervals $I'_{i} \subset I_{i}$ in such a way that the compact set $K_{i}$ is contained in the interior of the interval $I'_{i}$ and let $J'_{i}= \tilde{f}_{i}(I'_{i})$. Finally, use the isotopy extension property (Theorem 3.1 p. 185 in \cite{Hir}) to extend the map
$$ \begin{array}{rrcl}
\tilde{f}: & \bigcup_{i=1}^l I'_{i} & \rightarrow & \bigcup_{i=1}^l J'_{i} \\
 & x \in I_{i} & \mapsto & \tilde{f}_{i}(x)
\end{array} 
$$
to a diffeomorphism $g$ in $\mathrm{Diff}^{r}(M)$. In \cite{Hir}, the isotopy extension property is stated only for one disk but, with an induction, it is not difficult to prove this property for a union of disjoint closed disks (here closed intervals, which are one-dimensional disks).
\end{proof}

\section{Burnside property}

In this section, we prove Theorem \ref{Burn}. To prove it, we will use Definition \ref{intrinsic} of the group $\mathfrak{diff}^{1+Lip}(K)$. Hence we see our Cantor set $K$ as a subset of the real line $L=\mathbb{R}$.

Before starting the proof of Theorem \ref{Burn}, we need some definitions.

First, elements of $\mathfrak{diff}^{1}(K)$ have a well-defined derivative at each point of $K$. Indeed, fix a diffeomorphism $f$ in $\mathfrak{diff}^{1}(K)$ and a point $x_{0}$ of $K$. Then there exists an open interval $I$ of $\mathbb{R}$ which contains the point $x_{0}$ and a $C^{1}$-diffeomorphism $\tilde{f}: I \rightarrow \tilde{f}(I)$ such that $f_{I \cap K}=\tilde{f}_{I \cap K}$. Then the derivative $\tilde{f}'(x_{0})$ does not depend on the chosen extension $\tilde{f}$. Indeed,
$$\tilde{f}'(x_{0})= \lim_{\begin{array}{l} x \rightarrow x_{0} \\ x \in K \end{array}} \frac{f(x)-f(x_{0})}{x-x_{0}}$$
and the right-hand side of this equality depends only on $f$. We call this number the derivative of $f$ at $x_{0}$ and we denote it by $f'(x_{0})$.

As we can define the notion of derivative for the elements of our group, we also have a notion of hyperbolic fixed point which is recalled in the following definition.

\begin{definition}
Let $f$ be a diffeomorphism in $\mathfrak{diff}^{1}(K)$. Let $x_{0}$ be a point of $K$. We say that the point $x_{0}$ is a hyperbolic fixed point for $f$ if the following hold.
\begin{enumerate}
\item $f(x_{0})=x_{0}$.
\item $|f'(x_{0})| \neq 1$.
\end{enumerate}
\end{definition}

Observe that any diffeomorphism $f$ in $\mathfrak{diff}^{1}(K)$ with a hyperbolic fixed point $x_{0}$ is an infinite order element as the sequence $((f^{n})'(x_{0}))_{n \geq 0}$, which is a geometric sequence, is infinite.

The proof of Theorem \ref{Burn} relies on the following proposition, which is a consequence of a theorem by Sacksteder.

\begin{proposition} \label{SackCantor}
Let $G$ be a finitely generated subgroup of $\mathfrak{diff}^{1+Lip}(K)$. Then one of the following properties holds.
\begin{enumerate}
\item The group $G$ contains an element with a hyperbolic fixed point. 
\item Any minimal invariant subset for the action of $G$ on $K$ is finite.
\end{enumerate}
\end{proposition}

\subsection{Proof of Proposition \ref{SackCantor}}

Before proving the proposition, we have to recall the definition of a pseudogroup of diffeomorphisms of the real line $\mathbb{R}$.

\begin{definition} \label{pseudogroups}
A set $\mathcal{G}$ of $C^{r}$-diffeomorphisms  $g: \mathrm{dom}(g) \rightarrow \mathrm{ran}(g)$ between two open subsets $\mathrm{dom}(g)$ and $\mathrm{ran}(g)$ of $\mathbb{R}$ is called a pseudogroup of $C^{r}$-diffeomorphisms of $\mathbb{R}$ if
\begin{enumerate}
\item The set $\mathcal{G}$ is stable under composition, that is, if two elements $g$ and $h$ belong to $\\mathcal{G}$ with $\mathrm{ran}(h) \subset \mathrm{dom}(g)$, then the composition $gh$ of those elements belongs to $\mathcal{G}$.
\item The set $\mathcal{G}$ is stable under inverses, that is, for any element $g$ in $\mathcal{G}$, its inverse $g^{-1}$ belongs to $\mathcal{G}$.
\item The identity of $\mathbb{R}$ belongs to $\mathcal{G}$.
\item The set $\mathcal{G}$ is stable under restrictions, that is, if $g$ is an element of $\mathcal{G}$ and $A$ is an open subset of $\mathrm{dom}(g)$ then the diffeomorphism $g_{|A}:A \rightarrow g(A)$ belongs to $\mathcal{G}$.
\end{enumerate}
\end{definition}

Let $\mathcal{G}$ be a pseudogroup of diffeomorphisms of $\mathbb{R}$. A subset $A$ of $\mathbb{R}$ is \emph{invariant} under $\mathcal{G}$ if, for any point $x$ of $A$ and any element $g$ of $\mathcal{G}$, we have
$$ x \in \mathrm{ran}(g) \Rightarrow g(x) \in A.$$
A compact subset $A$ of $\mathbb{R}$ is a \emph{minimal invariant set} for the action of $\mathcal{G}$ if the set $A$ is nonempty and invariant under $\mathcal{G}$ and minimal for the inclusion relation among invariant nonempty compact subsets of $\mathbb{R}$.
A subset $\mathcal{S}$ of $\mathcal{G}$ is called a \emph{generating set} of $\mathcal{G}$ if any element $g$ of $\mathcal{G}$ is the restriction of a product of elements of $\mathcal{S}$. We say that a pseudogroup $\mathcal{G}$ is \emph{finitely-generated} if it admits a finite generating set.

Proposition \ref{SackCantor} is a consequence of the following theorem by Sacksteder (see \cite[Theorem 3.2.2]{Nav} for a proof).

\begin{theorem}[Sacksteder] \label{Sack}
Let $\mathcal{G}$ be a pseudogroup of $C^{1+Lip}$-diffeomorphisms of $\mathbb{R}$. Suppose that the following conditions hold.
\begin{enumerate}
\item The pseudogroup $\mathcal{G}$ has a Cantor set $K'$ as a minimal set of its action on the line. 
\item There exists a finite set $\mathcal{S}$ of elements of $\mathcal{G}$ which generates $\mathcal{G}$ whose domains have a compact intersection with $K'$.
\end{enumerate}
Then the pseudogroup $\mathcal{G}$ contains an element with a hyperbolic fixed point.
\end{theorem}

To apply this theorem, we need to connect the action of our group $G$ to a pseudogroup. This is achived by the following proposition, which is roughly a consequence of the equivalence between the two definitions of the group $\mathfrak{diff}^{r}(K)$.

\begin{proposition} \label{pseudo}
Let $G$ be a finitely generated subgroup of  $\mathfrak{diff}^{1+Lip}(K)$. Then there exists a pseudogroup $\mathcal{G}$ of $C^{1+Lip}$-diffemorphisms of the line $L$ such that
\begin{enumerate}
\item For any diffeomorphism $h$ of the pseudogroup $\mathcal{G}$ and for any point $x$ of $\mathrm{dom}(h) \cap K$, there exists an open neighbourhood $O$ of $x$ and an element $g$ of the group $G$ such that $h_{|O \cap K}=g_{|O \cap K}$.
\item For any element $g$ in $G$ and any point $x$ in $K$, there exists an open neighbourhood $O$ of $x$ and an element $h$ of $\mathcal{G}$ such that $O \subset \mathrm{dom}(h)$ and $g_{|O \cap K}=h_{|O \cap K}$.
\item There exists a finite generating set $\mathcal{S}$ of $\mathcal{G}$ such that, for any element $h$ in $\mathcal{S}$, the set $\mathrm{dom}(h) \cap K$ is compact.
\end{enumerate}
\end{proposition}

In particular, the pseudogroup $\mathcal{G}$ preserves $K$ and, for any point $x$ of $K$, the orbit of $x$ under the action of $G$ is also the orbit of $x$ under the action of $\mathcal{G}$. The last property will enable us to apply Sacksteder's theorem. Before proving the above proposition, we use it to prove Proposition \ref{SackCantor}.

\begin{proof}[Proof of Proposition \ref{SackCantor}]
Denote by $\mathcal{G}$ the pseudogroup associated to $G$ as in Proposition \ref{pseudo}. Let $K_{min} \subset K$ be a minimal invariant set for the action of $G$ on the Cantor set $K$. Then the set $K_{min}$ is a minimal invariant set for the action of the pseudogroup $\mathcal{G}$. Suppose that the set $K_{min}$ is infinite. Then the set $K'_{min}$ of accumulation points of $K_{min}$ is closed and nonempty. As the set $K'_{min}$ is also $G$-invariant, by minimality of $K_{min}$, we have $K'_{min}=K_{min}$ and the set $K_{min}$ is a Cantor set. By Theorem \ref{Sack}, there exist $x \in K_{min} \subset K$ and an element $h \in \mathcal{G}$ such that $x \in \mathrm{dom}(h)$ and $x$ is a hyperbolic fixed point for $h$. Hence, by Proposition \ref{pseudo}, some element of $G$ has a hyperbolic fixed point.
\end{proof}

\begin{proof}[Proof of Proposition \ref{pseudo}]

Let $g$ be a diffeomorphism in $\mathfrak{diff}^{1+ Lip}(K)$. By Definition \ref{intrinsic}, there exists a diffeomorphism $\hat{g}$ from an open neighbourhood $O_{g}$ of $K$ in $L$ to another neighbourhood of $K$ in $L$. For each connected component $O_{i,g}$ of $O_{g}$ (which is an open interval with compact intersection with $K$), let $\gamma_{i,g}= \hat{g}_{|O_{i,g}}$.
If $S$ is a symmetric finite generating set of $G$, let $\mathcal{S}$ be the finite set consisting of diffeomorphisms of the form $\gamma_{i,s}$, where $s \in S$. 

We claim that, if $\mathcal{G}$ denotes the pseudogroup generated by $\mathcal{S}$, then $\mathcal{G}$ satisfies the desired properties.

Let us prove it by induction on word length. Fix $x \in K$. For any element $g$ in the group $G$, let us denote by $l_{S}(g)$ the minimal number of factors required to write $g$ as a product of elements of $S$. Let us prove the following statement by induction on $n$: for any element $g \in G$ with $l_{S}(g) = n$, there exists an element $h$ of $\mathcal{G}$ such that $g=h_{|K}$ on a neighbourhood of $x$ in $K$.

For $n=1$, this property holds by construction of $\mathcal{S}$. Suppose this property is true for some $n$ and let us prove it for $n+1$. Let $g \in G$ and suppose $l_{S}(g)=n+1$. Write $g=sg_{1}$, where $l_{S}(g_{1})=n$ and $s$ belongs to $S$. By induction hypothesis, there exists $h \in \mathcal{G}$ such that $g_{1}=h$ on a neighbourhood of $x$. Let $O_{i,s}$ be the connected component of $O_{s}$ which contains $g_{1}(x)$. Then $g=\gamma_{i,s}h_{|K}$ on a neighbourhood of $x$ in $K$. This completes the induction.

For any element $h$ in $\mathcal{G}$ and any point $x$ in $\mathrm{dom}(h)$, we denote by $l_{\mathcal{S},x}(h)$ the minimal number of factors required to write $h$ as a product of elements of $\mathcal{S}$ on a neighbourhood of $x$. To finish the proof of Proposition \ref{pseudo}, it suffices to prove the following statement by induction on $n$. For any element $h$ in $\mathcal{G}$ and for any point $x$ in $\mathrm{dom}(h) \cap K$, if $l_{\mathcal{S},x}(g) \leq n$, then there exists an element $g$ of $G$ such that $g=h_{|K}$ on a neighbourhood of $x$ in $K$. This induction is straightforward to carry out and is left to the reader.
\end{proof}

\subsection{End of the proof of Theorem \ref{Burn}}

Let $G$ be a finitely generated periodic subgroup of $\mathfrak{diff}^{1+Lip}(K)$. Let us prove that $G$ is finite. We will use the following lemma.

\begin{lemma} \label{invneigh}
For any point $x$ in $K$, there exists a $G$-invariant clopen neighbourhood $U$ of $x$ such that the action of $G$ on $U$ factors through a finite group action.
\end{lemma}

Before proving this lemma, we use it to prove Theorem \ref{Burn}.

\begin{proof}[Proof of Theorem \ref{Burn}]
By Lemma \ref{invneigh}, there exists a cover $(U_{1},U_{2}, \ldots, U_{r})$ of $K$ by $G$-invariant clopen sets on which the action of $G$ factors through a finite group action. Changing $U_{i}$ to 
$$U_{i}- \bigcup_{1 \leq j \leq i-1} U_{j}$$ if necessary, we can suppose that the sets $U_{j}$ are pairwise disjoint. For any $i$, we denote by $G_{i}$ the group of restrictions to $U_{i}$ of elements of $G$. The groups $G_{i}$ are finite and the restriction maps define a morphism 
$$G \rightarrow \prod \limits _{i=1}^{r} G_{i}.$$
This morphism is one-to-one as the open sets $U_{i}$ cover $K$. Hence $G$ is finite.
\end{proof}

Now, we prove Lemma \ref{invneigh}. This lemma will be a consequence of the following lemma which will be proved afterwards.

\begin{lemma} \label{min}
For any minimal invariant subset $K_{min} \subset K$ for the action of $G$ on $K$, the set $K_{min}$ is finite and there exists a $G$-invariant clopen neighbourhood $U$ of $K_{min}$ on which the action factors through a finite group action.
\end{lemma}

\begin{proof}[Proof of Lemma \ref{invneigh}]
Take the closure $\overline{G.x}$ of the orbit of $x$ under the group $G$. Take a minimal set $K_{min}$ of the action of $G$ on the compact set $\overline{G.x}$. Apply Lemma \ref{min} to find a $G$-invariant clopen neighbourhood $U$ of $K_{min}$ on which the action factors through a finite group action. As $K_{min} \subset \overline{G.x}$, there exists $g$ in $G$ such that $g(x) \in U$. As $U$ is $G$-invariant, the point $x$ belongs to $U$ and the lemma is proved.
\end{proof}

We now prove Lemma \ref{min}.

\begin{proof}[Proof of Lemma \ref{min}] As elements of $G$ are finite order elements, no element of $G$ has a hyperbolic fixed point on $K$. By Proposition \ref{SackCantor}, the compact set $K_{min}$ has to be a finite set. Let $G_{2}$ be the subgroup of $G$ consisting of diffeomorphisms which pointwise fix $K_{min}$. It is a finite index subgroup of $G$ and it is finitely generated as a finite index subgroup of a finitely generated group. Notice that the derivative of any diffeomorphism in $G_{2}$ at each point of $K_{min}$ is either $1$ or $-1$ as we observed that the group $G$ contained no elements with a hyperbolic fixed point. Take the subgroup $G_{1}$ of $G_{2}$ consisting of elements whose derivative at each point of $K_{min}$ is $1$: it is still a finite index (finitely generated) normal subgroup of $G$.

\begin{claim} \label{Claim} Fix a point $x$ in $K_{min}$.  There exists a $G_{1}$-invariant clopen neighbourhood of $x$ on which $G_{1}$ acts trivially.
\end{claim}

Suppose this claim holds and let us see how to finish the proof of Lemma \ref{min}. 

By the claim, there exists a $G_{1}$-invariant clopen neighbourhood $U'$ of $x$ on which $G_{1}$ acts trivially. Observe that the open set
$$ U= \bigcup_{g \in G} g(U')$$
is a $G$-invariant clopen neighbourhood of $K_{min}$. As $G_{1}$ is a normal subgroup of $G$, it acts trivially on $U$. As $G_{1}$ is a finite index subgroup of $G$, the action of $G$ on $U$ factors through a finite group action. It remains to prve the claim.

\begin{proof}[Proof of Claim \ref{Claim}] Fix a finite generating set $S$ of $G_{1}$. By Definition \ref{intrinsic}, there exists an open interval $I$ of $\mathbb{R}$ which contains $x$ such that, for any element $s$ of the generating set $S$, there exists a diffeomorphism $\tilde{s} : I \rightarrow \tilde{s}(I)$ such that $s_{|I \cap K}=\tilde{s}_{|I \cap K}$. As the derivative of the diffeomorphism $\tilde{s}$ at $x$ is $1>0$, the diffeomorphism $\tilde{s}$ is orientation preserving.

We claim that any element of $G_{1}$ pointwise fixes $I \cap K$. Indeed, otherwise, there would exist a diffeomorphism $s$ in the generating set $S$ and a point $y$ in $I \cap K$ such that $s(y) \neq y$. Replacing $s$ with $s^{-1}$ if necessary, we can suppose that $s(y)$ lies in the interval of $\mathbb{R}$ delimited by $x$ and $y$. As $\tilde{s}$ is an orientation-preserving diffeomorphism between intervals of $\mathbb{R}$, it is strictly increasing and the sequence $(s^{n}(y))_{n \geq 0}=(\tilde{s}^{n}(y))_{n \geq 0}$ is infinite. This is a contradiction as elements of $G_{1}$ are finite order elements.
\end{proof}
\end{proof}

\subsection{A generalization of Theorem \ref{Burn}}

In the rest of the article, we will need the following generalization of Theorem \ref{Burn}.

\begin{proposition} \label{Burngen}
Let $F$ be a closed subset of $K$ and $G$ be a subgroup of $\mathfrak{diff}^{1+Lip}(K)$ which consists of elements which preserve $F$. Denote by $G(F)$ the group of restrictions to $F$ of elements of $G$. Suppose the group $G(F)$ is periodic. Then
\begin{enumerate}
\item The group $G(F)$ is finite.
\item Let $G_{1}$ be the subgroup of $G$ consisting of elements which pointwise fix $F$ and have a positive derivative at each point of $F$. Then $G_{1}$ is a finite index subgroup of $G$. 
\end{enumerate}
\end{proposition}

Observe that the second conclusion of this proposition implies the first one. Indeed, the second conclusion implies that the subgroup $H$ of $G$ consisting of diffeomorphisms which pointwise fix $F$ has a finite index and the group $G(F)$ is isomorphic to the quotient of $G$ by the subgroup $H$. However, in the following proof, we will first prove the first conclusion and then use this first conclusion to obtain the second conclusion.

As the proof of this proposition is sometimes really similar to the proof of Theorem \ref{Burn}, we will skip some details.

\begin{proof}[Proof of Proposition \ref{Burngen}]

First, let us prove by contradiction that any minimal invariant set for the action of $G(F)$ on $F$ is finite. Suppose the action of the group $G(F)$ on $F$ has an infinite minimal invariant subset $K_{1} \subset F$. Then the set $K_{1}$ has to be a Cantor set. Hence, by Proposition \ref{SackCantor}, the action of $G(F)$ on $K_{1}$ has a hyperbolic fixed point. It is impossible as the group $G(F)$ consists of finite order elements by hypothesis.

We then need the following claim.

\begin{claim}
The action of the group $G(F)$ on $F$ has only finite orbits.
\end{claim}

\begin{proof}
Suppose there exists a point $p$ of $F$ such that the orbit $G.p$ is infinite. Take a minimal invariant set $K_{2} \subset \overline{G.p}$ for the action of $G(F)$ on $F$. We just saw that the set $K_{2}$ has to be finite. Take the finite index subgroup $G_{2}$ of $G(F)$ consisting of elements which pointwise fix the finite set $K_{2}$ and have a positive derivative at each point of $K_{2}$. As in the proof of Theorem \ref{Burn}, we can prove that the group $G_{2}$ has to pointwise fix a neighbourhood of $K_{2}$ in $F$. This is not possible as $K_{2}$ is accumulated by an infinite orbit under $G$.
\end{proof}

Hence any orbit of the action of $G(F)$ on $F$ is finite. As in the proof of Theorem \ref{Burn} and as the group $G(F)$ is finitely generated, we obtain the following property. If a finitely generated subgroup of $G(F)$ pointwise fixes a point of $F$, then it pointwise fixes a neighbourhood of this point: otherwise, the group $G(F)$ would have an infinite orbit. Hence, using the compactness of $F$, we deduce that the group $G(F)$ is finite.

Let us prove the second point of the proposition. As the group $G(F)$ is finite, the subgroup $G_{3}$ of $G$ consisting of elements which pointwise fix $F$ is a finite index subgroup of $G$ : it is the kernel of the restriction morphism $G \rightarrow G(F)$. Hence it suffices to prove that the group $G_{1}$ is a finite index subgroup of $G_{3}$.

Observe that, as elements of $G_{3}$ pointwise fix $F$, the derivative of any element of $G_{3}$ at each accumulation point of $F$ is equal to one. Let us denote by $F'$ the set of accumulation points of $F$. As the group $G_{3}$ is finitely generated, there exists a neighbourhood $U$ of $F'$ such that the derivative of any element of $G_{3}$ is positive on $U$. Observe that the set $F \setminus U$ is compact and consists of isolated points : this set is finite. Moreover, the group $G_{1}$ is the kernel of the morphism
$$\begin{array}{rcl}
G_{3} & \rightarrow & \left\{ -1,1 \right\}^{F \setminus U} \\
g & \mapsto & (\mathrm{sgn}(g'(x)))_{x \in F \setminus U}
\end{array},$$
where, for any real number $\lambda \neq 0$, $\mathrm{sgn}(\lambda)=1$ if $\lambda >0$ and $\mathrm{sgn}(\lambda)=-1$ if $\lambda <0$. As the group $\left\{ -1,1 \right\}^{F \setminus U}$ is finite, the group $G_{1}$ is a finite index subgroup of $G_{3}$.
\end{proof}

\section{Groups without free subsemigroups on two generators}

In this section, we prove Theorem \ref{nosubsemi}. As in the proof of Theorem \ref{Burn}, we will use Definition \ref{intrinsic} of the group $\mathfrak{diff}^{1+Lip}(K)$. Hence we see our Cantor set $K$ as a subset of $\mathbb{R}$.

We fix a group $G$ satisfying the hypothesis of Theorem \ref{nosubsemi}: the subgroup $G$ of $\mathfrak{diff}^{1+Lip}(K)$ does not contain any free subsemigroup on two generators.

The proof is divided in three steps which correspond to subsections \ref{42}, \ref{43} and \ref{44}.
\begin{enumerate}
\item First, we find a finite index subgroup $G_{1}$ of $G$ such that any minimal invariant set for the action of $G_{1}$ on $K$ is a fixed point.
\item Then we prove that any element of the derived subgroup $G'_{1}$ of $G_{1}$ pointwise fixes a neighbourhood of $\mathrm{Fix}(G_{1})$. This is the main step of the proof which heavily relies on distortion estimates.
\item We deduce the theorem from the two above steps.
\end{enumerate}

The following subsection is devoted to a useful preliminary result.

\subsection{A preliminary result}

We will often need the following result. For any element $g$ in $\mathfrak{diff}^{1}(K)$, we denote by $\mathrm{Per}(g)$ the set of periodic points of $g$, \emph{i.e.} the set of points $p$ of $K$ such that there exists an integer $n \geq 1$ such that $g^{n}(p)=p$.

\begin{lemma} \label{periodicpoints}
For any element $g$ in $\mathfrak{diff}^{1}(K)$, there exists $N \geq 1$ such that
$$ \mathrm{Per}(g)= \left\{ p \in K, \ g^{N}(p)=p \right\}.$$
\end{lemma}

\begin{proof}
Fix an element $g$ in $\mathfrak{diff}^{1}(K)$ and define

$$\begin{array}{rrcl}
T: & \mathrm{Per}(g) & \rightarrow & \mathbb{R}_{+} \\
 & p & \mapsto & T(p)= \min \left\{ T \geq 1, \ g^{T}(p)= p \right\}

\end{array}
.$$

This lemma is a consequence of the two following claims.

\begin{claim} \label{locallybounded}
For any point $p$ in $\mathrm{Per}(g)$, there exists an open neighbourhood $U$ of the point $p$ such that
$U \cap \mathrm{Per}(g)= U \cap \overline{\mathrm{Per}(g)}$ and $T_{|U \cap \mathrm{Per}(g)}$ is bounded.
\end{claim}

\begin{claim} \label{closure}
$\overline{\mathrm{Per}(g)}=\mathrm{Per}(g)$.
\end{claim}

By Claim \ref{locallybounded} and Claim \ref{closure}, the set $\mathrm{Per}(g)$ is compact and the function $T$ is locally bounded on $\mathrm{Per}(g)$. Hence the function $T$ is bounded by an integer $M$. It suffices to take $N= M!$ to prove the lemma.
\end{proof}

\begin{proof}[Proof of Claim \ref{locallybounded}]
By definition of $T$, for any point $p$ of $\mathrm{Per}(g)$, $g^{T(p)}(p)=p$. Recall that the set $K$ is contained in $\mathbb{R}$. Fix a point $p$ in $\mathrm{Per}(g)$. By Definition \ref{intrinsic}, there exists an open interval $I'$ of $\mathbb{R}$ which contains the point $p$ and a homeomorphism $\tilde{h}_{1}: I' \rightarrow \tilde{h}_{1}(I)$ such that $\tilde{h}_{1|I' \cap K}=g^{T(p)}_{|I \cap K}$. The homeomorphism $\tilde{h}_{1}$ is not necessarily orientation-preserving but there exists an open interval $I$ of $\mathbb{R}$ which contains the point $p$ and an orientation-preserving homeomorphism $\tilde{h}:I \rightarrow \tilde{h}(I)$ such that $g^{2T(p)}_{|I \cap K}= \tilde{h}_{|I \cap K}$. 

Let $U= I \cap K$. We will prove that, for any point $x$ in $U$, either $g^{2T(p)}(x)=x$ or $x \notin \mathrm{Per}(g)$. This proves the claim as
$$ \mathrm{Per}(g) \cap U = \left\{ x \in U, \ g^{2T(p)}(x)=x \right\}$$
is closed in $U$ and $T$ is bounded by $2T(p)$ on $U$.

Take a point $x$ in $U$ and suppose that $g^{2T(p)}(x) \neq x$. Then, as $\tilde{h}$ is an increasing map which fixes the point $p$, either
$$ \left\{ g^{2nT(p)}(x), n >0 \right\}= \left\{ \tilde{h}^{n}(x), \ n>0 \right\}$$
is infinite and contained in $U$ or
$$ \left\{ g^{2nT(p)}(x), n <0 \right\}=\left\{ \tilde{h}^{n}(x), \ n<0 \right\}$$
is infinite and contained in $U$. In either cases $x \notin \mathrm{Per}(g)$.
\end{proof}

\begin{proof}[Proof of Claim \ref{closure}]
Let $p$ be a point of $\overline{\mathrm{Per}(g)}$. Then the closure of the orbit of $p$ under the action of the diffeomorphism $g$ contains a a minimal set $F$. By Proposition \ref{SackCantor}, this set $F$  has to be a periodic orbit of $g$ and $F \subset \mathrm{Per}(g)$: there is no periodic point of $g$ in a neighbourhood of a hyperbolic fixed point of $g$. Moreover, as, by Claim \ref{locallybounded}, the set $\mathrm{Per}(g)$ is open in $\overline{\mathrm{Per}(g)}$, there exists $n>0$ such that the point $g^{n}(p)$ belongs to $\mathrm{Per}(g)$. Hence the point $p$ belongs to $\mathrm{Per}(g)$. 
\end{proof}

\subsection{Definition of the finite index subgroup $G_{1}$} \label{42}

This section is devoted to the proof of the following proposition. For any subgroup $H$ of $\mathfrak{diff}^{1+Lip}(K)$, we denote by $\mathrm{Fix}(H)$ the subset of $K$ consisting of points which are fixed under all the elements of the group $H$.

\begin{proposition} \label{minpoint}
There exists a finite index subgroup $G_{1}$ of $G$ such that the two following properties hold.
\begin{enumerate}
\item Any minimal invariant subset for the action of $G_{1}$ on $K$ is a point of $\mathrm{Fix}(G_{1})$.
\item For any diffeomorphism $g$ in the group $G_{1}$ and any point $p$ of $\mathrm{Fix}(G_{1})$, we have
$$ g'(p)>0.$$
\end{enumerate}
\end{proposition}

We start the proof of this proposition by the following lemma, which is more or less a consequence of Sacksteder's Theorem.

\begin{lemma} \label{minimalfinite}
Any minimal invariant subset for the action of the group $G$ on $K$ is finite.
\end{lemma}

\begin{proof}

If one looks closely at the proof of Sacksteder's theorem, one can see that it implies directly that, if the action of the group $G$ on the Cantor set $K$ has an infinite minimal invariant set, then the group $G$ contains a free subsemigroup on two generators. We provide here a proof which uses Sacksteder's theorem.

Suppose that the action of the group $G$ on $K$ has an infinite minimal subset $K_{min}$. We want to prove that the group $G$ contains a free subsemigroup on two generators. This will give a contradiction and will complete the proof of the lemma. To do this, we use the following classical lemma (see \cite[Proposition 2 p.188]{dlH} for a proof).

\begin{lemma}[Positive ping-pong lemma] \label{positivepingpong}
Let $H$ be a group acting on a set $E$. Assume there exist elements $h_{1}$ and $h_{2}$ of $H$ as well as disjoint nonempty subsets $A$ and $B$ of $E$ such that
$$\left\{ \begin{array}{l}
h_{1}(A \cup B) \subset A \\
h_{2}(A \cup B) \subset B
\end{array}
\right.
.$$
Then the subsemigroup of $G$ generated by $h_{1}$ and $h_{2}$ is free.
\end{lemma}

First, as the set $K_{min}$ is infinite, it contains accumulation points. As the set of accumulation points of $K_{min}$ is closed and invariant under the action of the group $G$ and as $K_{min}$ is a minimal invariant set, we deduce that the set $K_{min}$ is a Cantor set. Then, by Proposition \ref{SackCantor}, there exists an element $h$ of $G$ with a hyperbolic fixed point $p \in K_{min}$. Taking $h^{-1}$ instead of $h$ if necessary, we can suppose that $ \left| h'(p) \right| <1$. Taking the definition of a diffeomorphism of $K$, we know that  there exists a $C^{1+Lip}$-diffeomorphism $\tilde{h}$ from an open interval $I$ of $\mathbb{R}$ which contains the point $p$ to an open interval $\tilde{h}(I)$ such that $\tilde{h}_{| I \cap K}=h_{| I \cap K}$. Moreover, choose the interval $I$ sufficiently small so that $\displaystyle \sup_{x \in I} | \tilde{h}'(x) | <1$. Hence the sequence of sets $(\tilde{h}^{n}(I))_{n \geq 0}$ is decreasing and
$$ \cap_{n \in \mathbb{N}} \tilde{h}^{n}(I)= \left\{ p \right\}.$$
As $K_{min}$ is a minimal Cantor set, there exists a point $p'$ in $G.p \cap (I \setminus \left\{ p \right\})$. Fix an element $g$ of the group $G$ such that $g(p)=p'$. Let $\tilde{g}$ be a diffeomorphism between two intervals of $\mathbb{R}$ which coincides with $g$ on a neighbourhood of $p$. Take an integer $N_{1}$ sufficiently large so that $\tilde{g} \circ \tilde{h}^{N_{1}}$ is defined on $I$, $\displaystyle \sup_{x \in I} |(\tilde{g} \circ \tilde{h}^{N_{1}})'(x) | <1$ and $g \circ h^{N_{1}}(I \cap K) \subset I \cap K$. The map $h_{1}=g \circ h^{N_{1}}_{| I \cap K}$ has a unique fixed point $p'$: indeed, the map $\tilde{g} \circ \tilde{h}^{N_{1}}$ has a fixed point on $I$ which is the limit of any positive orbit under the action of $\tilde{g} \circ \tilde{h}^{N_{1}}$. Hence this fixed point belongs to $K$. Observe that $h_{1}(p)=g(p) \neq p$. Hence there exists an open interval $J_{1}$ which contains the point $p$ such that $h_{1}(J_{1})$ is disjoint from $J_{1}$. Finally take an open interval $J_{2} \subset I$ which contains the point $p'$ and is disjoint from the interval $J_{1}$. Then there exist integers $N_{2} >0$ and $N_{3}>0$ such that
$$ h^{N_{2}}(J_{1} \cup J_{2}) \subset J_{1}$$
and
$$ h_{1}^{N_{3}}(J_{1} \cup J_{2}) \subset J_{2}.$$
By the positive ping-pong lemma, the group $G$ contains a free semigroup on two generators.
\end{proof}

Now, we can finish the proof of Proposition \ref{minpoint}

\begin{proof}[End of the proof of Proposition \ref{minpoint}]
By Lemma \ref{minimalfinite}, any minimal subset for the action of the group $G$ on the Cantor set $K$ is contained in the set
$$ F= \bigcap_{g \in G} \mathrm{Per}(g).$$
By Lemma \ref{periodicpoints}, the set $F$ is a closed subset of $K$. Moreover, it is invariant under the action of $G$. Let us denote by $G(F)$ the group of restrictions to $F$ of elements of $G$. By definition of $F$ and by Lemma \ref{periodicpoints}, the group $G(F)$ consists of finite order elements. By Proposition \ref{Burngen}, the group $G(F)$ is finite. Moreover, let $G_{1}$  be the subgroup of $G$ consisting of elements which pointwise fix $F$ and whose derivative at each point of $F$ is positive. Then the group $G_{1}$ is a finite index subgroup of $G$ by the same proposition. 

Let us check that this group $G_{1}$ satisfies the wanted property. Let $K_{min}$ be a minimal invariant subset of the action of the group $G_{1}$ on the Cantor set $K$. Then the set
$$ M=\bigcup_{g \in G} g(K_{min})$$
is a closed $G$-invariant subset of $K$ which consists of a finite number of copies of $K_{min}$. Any $G$-orbit of a point in this set $M$ is dense in $M$: it is a minimal subset for the action of $G$ on $K$. Hence, by Lemma \ref{minimalfinite},
$$ K_{min} \subset M \subset F \subset \mathrm{Fix}(G_{1}).$$
\end{proof}

\subsection{Behaviour of individual elements of $G'_{1}$} \label{43}

Let us fix a subgroup $G_{1}$ of $G$ which satisfies Proposition \ref{minpoint} for the rest of this section. We prove the following result. Recall that the derived subgroup $H'$ of a group $H$ is the subgroup of $H$ generated by the commutators of elements of $H$, \emph{i.e.} elements of the form $[h_{1},h_{2}]=h_{1}h_{2}h_{1}^{-1}h_{2}^{-1}$ with $h_{1}, \ h_{2} \in H$.

\begin{proposition} \label{accumulation}
Any element of the derived subgroup $G'_{1}$ of the group $G_{1}$ pointwise fixes a neighbourhood of $\mathrm{Fix}(G_{1})$.  
\end{proposition}

In this proposition, the neighbourhood can a priori depend on the chosen element of the group $G'_{1}$.

We split the proof of Proposition \ref{accumulation} into two steps.
\begin{enumerate}
\item We first prove that any point of $\mathrm{Fix}(G_{1})$ is accumulated by points of $\mathrm{Fix}(G'_{1})$.
\item Then we use this first step to prove Proposition \ref{accumulation}.
\end{enumerate}

Throughout the proof of the proposition, we will use the two following definitions.

\begin{definition}
Let $I$ be an interval of $\mathbb{R} \supset K$ and $f$ be an element of $\mathfrak{diff}^{1+Lip}(K)$. We say that
\begin{itemize}
\item $f$ is \emph{monotonous} on $I$ if there exists a $C^{1+Lip}$-diffeomorphism $\tilde{f}: I \rightarrow \tilde{f}(I)$ such that $f_{|I \cap K}=\tilde{f}_{I \cap K}$.
\item $f$ is \emph{increasing} on $I$ if there exists an orientation-preserving (\emph{i.e.} increasing) $C^{1+Lip}$-diffeomorphism $\tilde{f}: I \rightarrow \tilde{f}(I)$ such that $f_{|I \cap K}=\tilde{f}_{I \cap K}$.
\end{itemize}
\end{definition}

Let $A$ be a subset of $\mathbb{R}$ and $p$ be a point of $A$. We call \emph{left-neighbourhood} (respectively \emph{right-neighbourhood}) of the point $p$ in $A$ any subset of $\mathbb{R}$ which contains a set of the form $[p-\alpha,p) \cap A$ (resp. $(p,p+\alpha] \cap A$), for some $\alpha >0$. 

We say that the point $p$ is \emph{accumulated on the left} (respectively \emph{accumulated on the right}) by the set $A$ if, for any $\alpha>0$, $[p-\alpha,p) \cap A \neq \emptyset$ (resp. $(p,p+\alpha] \cap A \neq \emptyset$. Equivalently, a point $p$ is accumulated on the left (resp. right) by the set $A$ if any left-neighbourhood (resp. right-neighbourhood) of the point $p$ is nonempty. 

We say that the point $p$ is isolated on the left (resp. on the right) in $A$ if it is not accumulated on the left (resp. on the right) by the set $A$ or, equivalently, if it has a left-neighbourhood (resp. right-neighbourhood) in $A$ which is empty.

\subsubsection{First step}

We formulate this first step as a proposition.

\begin{proposition} \label{firststep}
Let $p$ be a point of $\mathrm{Fix}(G_{1})$.

If the point $p$ is accumulated on the left (respectively on the right) by points of $K$, then the point $p$ is accumulated on the left (resp. on the right) by points of $\mathrm{Fix}(G'_{1})$.
\end{proposition}

\begin{proof}[Proof of Proposition \ref{firststep}]
Observe that it suffices to prove the following property.
Any point of $\mathrm{Fix}(G_{1})$ which is isolated on the left (respectively on the right) in $\mathrm{Fix}(G_{1})$ but not in $K$ is accumulated on the left (resp. on the right) by points of $\mathrm{Fix}(G'_{1})$.
As the proof of the property for points which are isolated on the right is similar, we will prove the property only for points which are isolated on the left. Hence let $p$ be a point of $\mathrm{Fix}(G_{1})$ which is isolated on the left in $\mathrm{Fix}(G_{1})$ but not in $K$. 

To start the proof, we will define a ``local minimal invariant subset'' $K_{1}$ in a left-neighbourhood of $p$. The subset $K_{1}$ accumulates to $p$ and we will see later that any element of $G'_{1}$ fixes the points of $K_{1}$ in a left-neighbourhood of $p$.

As $G_{1}$ is a finite index subgroup of $G$ and the group $G$ is finitely generated, then the subgroup $G_{1}$ is finitely generated. Fix a finite symmetric generating set $S$ of $G_{1}$ and $\alpha_{0} >0$ small enough so that $\mathrm{Fix}(G_{1}) \cap [p-\alpha_{0},p)= \emptyset$ and, for any element $s$ of $S$, the diffeomorphism $s$ is monotonous on $[p-\alpha_{0},p]$. Then any element of $S$ is increasing on $[p-\alpha_{0},p]$ as, by the choice of $G_1$, any element of $G_{1}$ has a positive derivative at the point $p$.

Consider the set $\mathcal{M}$ of closed nonempty subsets $A$ of $(p-\alpha_{0},p)\cap K$ such that, for any element $s$ of the generating set $S$
$$ s(A) \cap (p-\alpha_{0},p) \subset A.$$
This last property is an analogue of a ``local invariance'' property. Of course, the set $\mathcal{M}$ is nonempty as the subset $(p-\alpha_{0},p) \cap K$ belongs to $\mathcal{M}$. 

Take a point $p_{1}$ in $(p-\alpha_{0},p) \cap K$. Let 
$$ p_{2}= \max \left\{ s(p_{1}), \ s \in S \right\}.$$
Observe that, necessarily, $p_{2}>p_{1}$: otherwise, the point $p_{2}$ would be fixed under any element of $S$ hence any element of $G_{1}$, in contradiction with the definition of $\alpha_{0}$. To define the subset $K_{1}$, we need the following lemma.

\begin{lemma} \label{inductive}
For any $A$ in $\mathcal{M}$,
$$ A \cap [p_{1},p_{2}] \neq \emptyset$$
and the point $p$ is accumulated by points of $A$. 
\end{lemma}

Before proving Lemma \ref{inductive}, let us see how to construct the subset $K_{1}$. 

The set $\mathcal{M}$ is partially ordered by the set inclusion relation. Moreover, by compactness and Lemma \ref{inductive}, for any totally ordered family $(A_{i})_{i \in I}$ of elements of $\mathcal{M}$, the set
$$ \bigcap_{i \in I} A_{i}$$
is nonempty and belongs to $\mathcal{M}$: this set is a lower bound for this totally ordered family. By Zorn's lemma, the set $\mathcal{M}$ contains a minimal element for the inclusion relation. We denote by $K_{1}$ this element of $\mathcal{M}$. We can see it as a minimal invariant set for the left-germ of $G_{1}$ at $p$.

Now, let us prove Lemma \ref{inductive}.

\begin{proof}[Proof of Lemma \ref{inductive}]
Denote by $p'$ a point in $A \cap (p-\alpha_{0},p)$.

If the point $p'$ belongs to $[p_{1},p_{2}]$, there is nothing to prove.

Suppose now that the point $p'$ belongs to the interval $(p_{2},p)$. Denote by $\mathcal{B}$ the set of elements $g$ of $G_{1}$   with the following property. There exists a family $(s_{i})_{1 \leq i \leq p}$ of elements of $S$ such that $g=s_{1}s_{2} \ldots s_{p}$ and 
$$ \forall 1 \leq i \leq p, \  s_{i}s_{i+1} \ldots s_{p}([p',p] \cap K) \subset (p-\alpha_{0},p].$$
Now let $B$ be the subset of $K \cap (p-\alpha_{0},p)$ defined by 
$$B= \left\{ b(p'), b \in \mathcal{B} \right\}.$$
Observe that, as the subset $A$ belongs to $\mathcal{M}$, the subset $B$ is contained in $A$. 

Let $m=\inf(B \cap [p_{1},p])$. It suffices to prove that the element $m$ belongs to $[p_{1},p_{2})$. We will do it by contradiction. Suppose that $p_{2} \leq m$. Then, as the point $m$ does not belong to $\mathrm{Fix}(G_{1})$, there exists an element $s$ of $S$ such that $s(m) < m$ (recall that the set $S$ is symmetric). As $p_{2} \leq m$ and by construction of the point $p_{2}$, we have $p_{1} \leq s(m)$ and $p-\alpha_{0} < s(m)$. Hence, if $x$ is a point of $B$ close to $m$, then $s(x)$ is also a point of $B$ which belongs to $[s(m),m)$, in contradiction with the definition of $m$.

If $p-\alpha_{0} < p'<p_{1}$, the proof is analogous. Namely look at the same set $B$ but look at the supremum of $B \cap [p_{1},p_{2}]$ instead of the infimum and prove that it belongs to $(p_{1},p_{2}]$.

Observe that we can take the point $p_{1}$ as close as we want to the point $p$. Hence the point $p$ is accumulated by points of $A$.
\end{proof}

Let $H$ be any subgroup of $\mathfrak{diff}^{1}(K)$, $F$ be a closed subset of $K$ and $q$ be a point of $\mathrm{Fix}(H)$ which is accumulated on the left by $F$ such that any element of $H$ has a positive derivative at $q$. Define
$$H_{F,q,-}= \left\{ g \in H \ | \ \exists \alpha>0, \ g_{F\cap [q-\alpha,q]}= Id_{F\cap [q-\alpha,q]} \right\}.$$
Observe that $H_{F,q,-}$ is a normal subgroup of $H$. We denote by $\mathcal{H}(F,q)$ the group $H /  H_{F,q,-}$. This is the ``group of left-germs at $q$ of elements of $H$ restricted to $F$''. In particular we set $\mathcal{G}_{1}(K_{1},p)=G_{1} /  G_{1 \ K_{1},p,-}$. This is the ``group of left-germs at $p$ of elements of $G_{1}$ restricted to $K_{1}$''.

The following proposition completes the proof of the first step. \end{proof}

\begin{proposition} \label{minabelian}
\begin{enumerate}
\item Any element of $G_{1}$ either fixes all the points of $K_{1}$ or has no fixed point on a left-neighbourhood of $p$ (which depends on the element of $G_{1}$).
\item The group $\mathcal{G}_{1}(K_{1},p)$ is abelian.
\end{enumerate}
\end{proposition}

\begin{remark} Actually, the first statement of this proposition implies the second one, as we will see during the proof of this proposition. \end{remark}

We obtain immediately the following corollary which will be useful later.

\begin{corollary} \label{nofixedpoint} 
Let $p$ be a point of $\mathrm{Fix}(G_{1})$ which is isolated on the left in $\mathrm{Fix}(G_{1})$ but accumulated on the left by $K$. Then there exists an element $h$ of $G_{1}$ and $\alpha >0$ such that
\begin{enumerate}
\item the diffeomorphism $h$ is increasing on $[p-\alpha,p]$.
\item $\forall x \in [p-\alpha,p] \cap K, h(x)>x$.
\end{enumerate}
\end{corollary}

Of course, we have analogous statements for points which are isolated on the right in $\mathrm{Fix}(G_{1})$.

To prove Proposition \ref{minabelian}, we need the following lemma.

\begin{lemma} \label{Holder}
Let $H$ be a subgroup of $\mathfrak{diff}^{1}(K)$, $F$ be a closed subset of $K$ and $q$ be a point of $\mathrm{Fix}(H)$ which is accumulated on the left by $F$ such that any element of $H$ has a positive derivative at $q$..
Suppose that there exists $\alpha_{0}>0$ such that, for any diffeomorphism $h$ in $H$, the following property holds.

For any element $h$ of the group $H$ and any $\alpha \in (0,\alpha_{0})$, if the diffeomorphism $h$ is increasing on $[q-\alpha,q]$ and if the diffeomorphism $h$ has a fixed point $x$ in $[q-\alpha,q] \cap F$ then the diffeomorphism $h$ pointwise fixes $[x,q] \cap F$.

Then the group $\mathcal{H}(F,q)$ is abelian.
\end{lemma}

To prove this lemma, we use the same techniques as in the proof of a famous theorem by H\"older, which states that any group of fixed-point-free homeomorphisms of the real line is abelian. We need the following definition for this proof.

Observe that we could have thought of a weaker and more natural hypothesis for Lemma \ref{Holder}: any element of $\mathcal{H}(F,q)$ has a representative with no fixed point in $V \setminus \left\{ q \right\}$, where $V$ is a left-neighbourhood of $q$ in $F$. However, our proof does not work with this weaker hypothesis. This is due to the various speeds of convergence of the orbits to the point $q$ that can exist. The hypothesis of Lemma \ref{Holder} that we took avoids this problem.

\begin{definition}
Let $H$ be a group and $\preceq$ be an order on $H$. The order $\preceq$ is called
\begin{enumerate}
\item \emph{total} if, for any elements $h_{1}$ and $h_{2}$ of $H$, either $h_{1} \preceq h_{2}$ or $h_{2} \preceq h_{1}$.
\item \emph{biinvariant} if, for any elements $h_{1}$, $h_{2}$ and $h_{3}$ of $H$,
$$ h_{1} \preceq h_{2} \Rightarrow h_{3}h_{1} \preceq h_{3}h_{2}$$
and
$$ h_{1} \preceq h_{2} \Rightarrow h_{1}h_{3} \preceq h_{2}h_{3}.$$
\item \emph{Archimedean} if, for any element $h_{1}$ and $h_{2}$ of $H$ with $1 \prec h_{2}$ (\emph{i.e.} $1 \preceq h_{2}$ and $h_{2} \neq 1$), there exists an integer $n \geq 0$ such that $h_{1} \preceq h_{2}^{n}$. 
\end{enumerate}
\end{definition}

The main idea of the proof of Lemma \ref{Holder} is to apply the following lemma (see \cite[Proposition 2.2.29]{Nav} for a proof of this lemma).

\begin{lemma} \label{Archimedean}
Any group which admits a bi-invariant total Archimedean order is a subgroup of $(\mathbb{R},+)$.
\end{lemma}

\begin{proof}[Proof of Lemma \ref{Holder}]
We define an order $\preceq$ on the group $\mathcal{H}(F,q)$ in the following way. For any elements $\xi$ and $\eta$ which are respectively represented by elements $g_{\xi}$ and $g_{\eta}$ of $H$, we have $\xi \preceq \eta$ if and only if there exists $\alpha >0$ such that, for any point $x$ of $F \cap (q-\alpha,q)$, we have $g_{\xi}(x) \leq g_{\eta}(x)$. We will prove that this defines a biinvariant total Archimedean order on $\mathcal{H}(F,q)$. By Lemma \ref{Archimedean}, this implies that the group $\mathcal{H}(F,q)$ is abelian and proves Lemma \ref{Holder}.

Let $\xi_{1}$, $\xi_{2}$ and $\xi_{3}$ be elements of $\mathcal{H}(F,q)$ with $\xi_{1} \preceq \xi_{2}$. Let us prove that $\xi_{1} \xi_{3} \preceq \xi_{2} \xi_{3}$ and that $\xi_{3} \xi_{1} \preceq \xi_{3} \xi_{1}$.

Take elements $g_{1}$, $g_{2}$ and $g_{3}$ of the group $H$ which respectively represent the elements $\xi_{1}$, $\xi_{2}$ and $\xi_{3}$. Take $\alpha >0$ small enough such that
\begin{enumerate}
\item The diffeomorphisms $g_{1}$, $g_{2}$ and $g_{3}$ are increasing on $[q-\alpha,q]$.
\item For any point $x$ in $[q-\alpha,q) \cap F$, $g_{1}(x) \leq g_{2}(x)$.
\end{enumerate} 
Take $\alpha' >0$ small enough so that $g_{3}([q-\alpha',q] \cap F) \subset [q-\alpha,q] \cap F$. Then, for any point $x$ in $[q-\alpha',q] \cap F$,
$$g_{1}(g_{3}(x)) \leq g_{2}(g_{3}(x))$$
and $\xi_{1} \xi_{3} \preceq \xi_{2} \xi_{3}$.
Now, take $\alpha''>0$ small enough such that  $g_{1}([q-\alpha'',q] \cap F) \subset [q-\alpha,q] \cap F$ and  $g_{2}([q-\alpha'',q] \cap F) \subset [q-\alpha,q] \cap F$. Then, as the diffeomorphism $g_{3}$ is increasing on $[q-\alpha,q]$, for any point $x$ of $[q-\alpha'',q] \cap F$,
$$ g_{3}(g_{1}(x)) \leq g_{3}(g_{2}(x))$$
and $\xi_{3} \xi_{1} \preceq \xi_{3} \xi_{2}$.

We have just proved that the order $\preceq$ is biinvariant. Let us explain why this order is total. Let $\xi$ and $\eta$ be elements of $\mathcal{H}(F,q)$ which are respectively represented by elements $g_{\xi}$ and $g_{\eta}$ of $H$. By hypothesis of Lemma \ref{Holder}, either the diffeomorphism $g_{\xi} \circ g_{\eta}^{-1}$ is equal to the identity on $F \cap (q-\alpha,q)$ for $\alpha>0$ small enough or this diffeomorphism displaces all the points of $F$ in a left-neighbourhood of $q$. Hence either $\xi \eta^{-1} \succeq 1$ or $\xi \eta^{-1} \preceq 1$. By invariance of the order $\preceq$ under right-multiplication, we deduce that either $\xi \succeq \eta$ or $\xi \preceq \eta$. The order $\preceq$ is total. 

Now, let us prove that it is Archimedean. Let $\xi_{1}$ and $\xi_{2}$ be elements of $\mathcal{H}(F,q)$ such that $\xi_{1} \succ 1$. If $\xi_{2} \preceq 1$, then $\xi_{2} \preceq 1 \prec \xi_{1}$ so we suppose that $\xi_{2} \succ 1$ in what follows.

Take respective representatives $g_{1}$ and $g_{2}$ of $\xi_{1}$ and $\xi_{2}$ in $H$.
Take $\alpha \in (0,\alpha_{0})$ small enough such that
\begin{enumerate}
\item The diffeomorphisms $g_{1}$ and $g_{2}$ are increasing on $[q-\alpha,q]$.
\item For any point $x$ in $[q-\alpha,q) \cap F$, $g_{1}(x) >x$ and $g_{2}(x) >x$.
\end{enumerate}
Fix a point $x_{0}$ in $F \cap [q-\alpha,q)$. Observe that any positive orbit under $g_{1}$ of points of $F \cap [q-\alpha,q)$ converges to the point $q$. Hence there exists $k>0$ such that
$g_{2}(x_{0}) < g_{1}^{k}(x_{0})$, which can be rewritten
$$g_{2}(x_{0}) < g_{1}^{k}g_{2}^{-1}(g_{2}(x_{0})).$$
As the diffeomorphism $g_{1}^{k}g_{2}^{-1}$ is increasing on $[g_{2}(x_{0}),q]$ then, by hypothesis of Lemma \ref{Holder}, one of the following occurs.
\begin{enumerate}
\item Either the diffeomorphism $g_{1}^{k}g_{2}^{-1}$ pointwise fixes a left-neighbourhood of the point $q$ in $F$. In this case $\xi_{1}^{k} \xi_{2}^{-1}=1$.
\item  Or it has no fixed point in $[g_{2}(x_{0}),q] \cap F$ in which case
$$\forall x \in F \cap [g_{2}(x_{0}),p], \ g_{1}^{k} \circ g_{2}^{-1}(x) >x.$$
\end{enumerate}
In either case, $\xi_{1}^{k} \xi_{2}^{-1} \succeq 1$ and, by invariance of the relation $\preceq$ under right-multiplication, $\xi_{1}^{k} \succeq \xi_{2}$. The order $\preceq$ is Archimedean.
\end{proof}

\begin{proof}[Proof of Proposition \ref{minabelian}]
We denote by $K'_{1}$ the set of accumulation points of $K_{1}$. By minimality of $K_{1}$, observe that either $K'_{1} \cap (p-\alpha_{0}, p)= \emptyset$ or $K'_{1} \cap (p-\alpha_{0},p)= K_{1}$. Indeed, if the set $K'_{1}\cap (p-\alpha_{0},p)$ is nonempty, then $K'_{1} \cap (p-\alpha_{0},p)$ is an element of $\mathcal{M}$ which is contained in $K_{1}$, hence which is equal to $K_{1}$ by minimality of $K_{1}$. This remark splits the proof into two cases.

\textbf{First case}: $K'_{1}\cap (p-\alpha_{0},p)= \emptyset$. As any point of the set $K_{1}$ is isolated, any element of $G_{1}$ either pointwise fixes all the points of $K_{1}$ or displaces all the points of $K_{1}$ on a left neighbourhood of the point $p$ (neighbourhood which a priori depends on the element of $G_{1}$). Lemma \ref{Holder} implies that the group $\mathcal{G}_{1}(K_{1},p)$ is abelian.

\textbf{Second case}: $K'_{1} \cap (p-\alpha_{0}, p)=K_{1}$. In this second case, it also suffices to prove that any element of $G_{1}$ satisfies the hypothesis of Lemma \ref{Holder}. However, proving this fact is harder than in the first case.

Suppose for a contradiction that there exists $\alpha>0$ and a diffeomorphism $g_{1}$ in $G_{1}$ such that
\begin{enumerate}
\item The diffeomorphism $g_{1}$ is increasing on $[p-\alpha,p]$.
\item There exists a point $x$ in $K_{1} \cap [p-\alpha,p)$ such that $g_{1}(x)=x$.
\item The diffeomorphism $g_{1}$ does not fix all the points of $[x,p] \cap K_{1}$.
\end{enumerate} 

Let $p_{1}<p_{2}$ be two points of $[x,p] \cap \mathrm{Fix}(g_{1})$ such that $K_{1} \cap (p_{1},p_{2}) \neq \emptyset$ and, for any point $x$ in $K_{1} \cap (p_{1},p_{2})$, $g_{1}(x) \neq x$. In this second case, Proposition \ref{minabelian} is a consequence of the two following lemmas.

\begin{lemma} \label{crossing}
There exists an element $h$ of $G_{1}$ such that $p_{1} < h(p_{1}) < p_{2}$ and the diffeomorphism $h$ is increasing on $[p_{1},p]$.
\end{lemma}

Before stating the second lemma, we need a definition which is a generalization of a standard definition for pseudogroups or groups of homeomorphisms of the real line (see \cite[Definition 2.2.43]{Nav}).

\begin{definition}
Two elements $g$ and $h$ of $\mathfrak{diff}^{1+Lip}(K)$ are \emph{crossed} if there exists points $p_{1}<p_{2}$ of the Cantor set $K$ such that
\begin{enumerate}
\item The diffeomorphisms $g$ and $h$ are increasing on $[p_{1},p_{2}]$.
\item The points $p_{1}$ and $p_{2}$ are fixed under $g$ but the diffeomorphism $g$ has no fixed point in $(p_{1},p_{2})$.
\item Either we have $p_{1}<h(p_{1})<p_{2}$ or $p_{1}<h(p_{2})<p_{2}$.
\end{enumerate}
\end{definition}

Observe that the element $g_{1}$ which is defined above and the element $h$ which is given by Lemma \ref{crossing} are crossed. Now, it suffices to use Lemma \ref{crossed} below to obtain the wanted contradiction.

\begin{lemma} \label{crossed}
If two elements $g$ and $h$ of $\mathfrak{diff}^{1+Lip}(K)$ are crossed, then the subgroup generated by $g$ and $h$ contains a free semi-group on two generators.
\end{lemma}
\end{proof}

Now, let us prove Lemmas \ref{crossing} and \ref{crossed}.

\begin{proof}[Proof of Lemma \ref{crossing}] Denote by $\mathcal{B}$ the set of elements $g$ of $G_{1}$  with the following property. There exists a family $(s_{i})_{1 \leq i \leq p}$ of elements of $S$ such that $g=s_{1}s_{2} \ldots s_{p}$ and 
$$ \forall 1 \leq i \leq p, \  s_{i}s_{i+1} \ldots s_{p}([p_{1},p] \cap K) \subset (p-\alpha_{0},p].$$
Now let $A$ be the closure in $K_{1}$ of the subset of $K \cap (p-\alpha_{0},p)$ defined by 
$$B=\left\{ b(p_{1}), b \in \mathcal{B} \right\}.$$
Observe that, as the subset $K_{1}$ belongs to the collection $\mathcal{M}$, we have $A \subset K_{1}$. Moreover, the set $A$ belongs to the collection $\mathcal{M}$ so that $A=K_{1}$, by minimality of $K_{1}$. Hence the set $B$ is dense in $K_{1}$ and $B \cap(p_{1},p_{2}) \neq \emptyset$ as $K_{1} \cap (p_{1},p_{2}) \neq \emptyset$. This proves Lemma \ref{crossing}.
\end{proof}

The proof of Lemma \ref{crossed} is similar to \cite[Lemma 2.2.44]{Nav}.

\begin{proof}[Proof of Lemma \ref{crossed}]
Suppose for instance that $p_{1}< h(p_{1}) < p_{2}$. Moreover, as the diffeomorphism $g$ is increasing on $[p_{1},p_{2}]$, taking $g^{-1}$ instead of $g$ if necessary, we can suppose that
$$ \forall x \in (p_{1},p_{2}) \cap K, \ g(x) < x.$$
Observe that, for any point $x$ in $(p_{1},p_{2}) \cap K$, the sequence $g^{n}(x)$ converges to the point $p_{1}$. In particular, the point $p_{1}$ is accumulated on the right by points of $K$. Take $\alpha>0$ small enough so that
\begin{enumerate}
\item $p_{1} + \alpha \in K$ and the diffeomorphism $h$ is monotonous on $[p_{1},p_{1}+\alpha]$.
\item $h(p_{1}+ \alpha) < p_{2}$.
\item The sets $[p_{1},p_{1}+\alpha]$ and $h([p_{1},p_{1}+\alpha] \cap K)$ are disjoint.
\end{enumerate}
Take a sufficiently large integer $n$ such that $g^{n}(h(p_{1}+\alpha)) <p_{1}+\alpha$. Let $f_{1}=g^{n}$ and $f_{2}=h \circ g^{n}$. Observe that
$$ f_{1}([p_{1},p_{1}+\alpha]\cap K \cup h([p_{1},p_{1}+\alpha]\cap K)) \subset [p_{1},p_{1}+\alpha] \cap K$$
and that
$$ f_{2}([p_{1},p_{1}+\alpha] \cap K \cup h([p_{1},p_{1}+\alpha] \cap K)) \subset h([p_{1},p_{1}+\alpha] \cap K).$$
Now, by the positive ping-pong lemma (Lemma \ref{positivepingpong}), the semigroup generated by $f_{1}$ and $f_{2}$ is free.
\end{proof}

\subsubsection{Second step}

Now, we are ready for the second step of the proof of Proposition \ref{accumulation}. We will reformulate this second step as two propositions. Fix a point $p$ in $\mathrm{Fix}(G_{1})$ which is accumulated on the left by the set $K$. Denote by $\mathcal{G}_{1,p}$ the group of left-germs at $p$ of elements of $G_{1}$: this the quotient of the group $G_{1}$ by the normal subgroup
$$G_{1, p,-}= \left\{ g \in G_1 \ | \ \exists \alpha>0, \ g_{K\cap [p-\alpha,p]}= Id_{F\cap [p-\alpha,p]} \right\}.$$
 
\begin{proposition} \label{outsidemin}
Suppose that the point $p$ is isolated on the left in $\mathrm{Fix}(G_{1})$. Then the group $\mathcal{G}_{1,p}$ is abelian.
\end{proposition}

Of course, we have an analogous statement in the case where the point $p$ is isolated on the right in $\mathrm{Fix}(G_{1})$.

In the case where the point $p$ is accumulated on the left by points of $\mathrm{Fix}(G_{1})$, we will prove the following stronger proposition.

\begin{proposition} \label{outsideminacc}
Suppose the point $p$ is accumulated on the left by points of $\mathrm{Fix}(G_{1})$. Then there exists a left-neighbourhood $L_{p}$ of the point $p$ with the following properties.
\begin{enumerate}
\item $L_{p} \cap K$ is invariant under the action of $G_{1}$.
\item The action of $G'_{1}$ on $L_{p} \cap K$ is trivial.
\end{enumerate}
\end{proposition}

Once again, there is an analogous statement for points which are accumulated on the right by points of $\mathrm{Fix}(G_{1})$.

The two above propositions and their variants for right-neighbourhoods imply Proposition \ref{accumulation}.

We will start by proving Proposition \ref{outsidemin}. Then we will prove Proposition \ref{outsideminacc}. 

To prove Proposition \ref{outsidemin}, we need the following lemma, which is a variant of Kopell lemma for diffeomorphisms of the half-line (see \cite{Kop} to see this lemma and its proof).

\begin{lemma} \label{Kopell}
Let $g_{1}$ and $g_{2}$ be elements of $\mathfrak{diff}^{1+Lip}(K)$. Let $p \in \mathrm{Fix}(g_{1}) \cap \mathrm{Fix}(g_{2})$ be a point accumulated on the left by points of $K$ such that the elements $g_{1}$ and $g_{2}$ have a positive derivative at $p$.
Suppose that
\begin{enumerate}
\item There exists $\alpha>0$ such that $(p-\alpha,p) \cap \mathrm{Fix}(g_{1})= \emptyset$ and such that $g_{2}$ is increasing on $(p-\alpha,p)$.
\item For any point $x$ in $(p-\alpha,p)\cap K$, $g_{1}g_{2}(x)=g_{2}g_{1}(x)$.
\end{enumerate}
Then $(p-\alpha,p) \cap \mathrm{Fix}(g_{2}) = \emptyset$ or the diffeomorphism $g_{2}$ fixes all the points of $K$ in a left-neighbourhood of the point $p$.
\end{lemma}

Of course, we have an analogous lemma for fixed points of $g_{1}$ which are isolated on the right. In the classical Kopell lemma, we only need that $g_{1}$ is $C^{1+bv}$ and $g_{2}$ is $C^{1}$. In contrast, in our case, we can lower the regularity to $C^{1+bv}$ but we still need both elements $g_{1}$ and $g_{2}$ to have a $C^{1+bv}$ regularity. The proof of this lemma is closely related to the proof in the standard case.

\begin{proof}[Proof of Lemma \ref{Kopell}]
Take $\alpha>0$ such that $g_{1}$ has no fixed point on $[p-\alpha,p)$ and such that there exist $C^{1+Lip}$-diffeomorphisms $\tilde{g}_{1}$ and $\tilde{g}_{2}$ defined on $[p-\alpha,p]$ such that $g_{1 \ | [p-\alpha,p] \cap K}= \tilde{g}_{1 \ | [p-\alpha,p] \cap K}$ and $g_{2 \ | [p-\alpha,p] \cap K}= \tilde{g}_{2 \ | [p-\alpha,p] \cap K}$. For $i=1,2$, denote by $k_{i}$ the Lipschitz constant of $\log(| \tilde{g}_{i}'|)_{|[p-\alpha,p]}$ and by $D$ the diameter of the compact subset $K$ of $\mathbb{R}$.

Suppose that $g_{2}$ has a fixed point $p'$ in $[p-\alpha,p)$. We will prove that there exists $M>0$ such that, for any $k>0$, $ \sup_{x \in [p',p] \cap K} |( g_{2}^k)'(x) | \leq M$. From this, we will deduce that $g_{2}$ fixes the points of $[p',p] \cap K$.

Taking $g_{1}^{-1}$ instead of $g_{1}$ if necessary, we can suppose that, for any point $x$ in $[p-\alpha,p) \cap K$, $\tilde{g}_{1}(x)>x$. As the diffeomorphisms $g_{1}$ and $g_{2}$ commute the points $g_{1}^{n}(p')$, for $n \geq 0$ are fixed points of $g_{2}$ and form a sequence which converges to the point $p$. Hence $g_{2}'(p)=1$.

For any $k \geq 0$ and any $n \geq 0$, $g_{2}^{k}=g_{1}^{-n}g_{2}^{k}g_{1}^{n}$. Hence, for any point $x$ of $K \cap [p',p]$, $$(*) \ \ (g_{2}^{k})'(x)= \frac{(g_{1}^{n})'(x)}{(g_{1}^{n})'(g_{2}^{k}(x))}.(g_{2}^{k})'(g_{1}^{n}(x)).$$
Observe that the sequence $((g_{2}^{k})'(g_{1}^{n}(x)))_{n}$ converges to $1=(g_{2}^{k})'(p)$ as $n \rightarrow + \infty$. Let us prove that
$$ \frac{(g_{1}^{n})'(x)}{(g_{1}^{n})'(g_{2}^{k}(x))} \leq e^{k_{1}D}.$$
Indeed,
$$ \begin{array}{rcl}
 \log(|\frac{(g_{1}^{n})'(x)}{(g_{1}^{n})'(g_{2}^{k}(x))}|) & = & \displaystyle \sum_{i=0}^{n-1} \log(|g_{1}'(g_{1}^{i}(x))|)-\log(|g_{1}'(g_{1}^{i}(g_{2}^{k}(x)))|).
 \end{array}$$
 For any index $i \geq 0$, denote by $I_{i}$ the closed interval of $\mathbb{R}$ whose ends are $g_{1}^{i}(x)$ and $g_{1}^{i}(g_{2}^{k}(x))$ and denote by $n_{0} \geq 0$ the integer such that the point $x$ belongs to the interval $[g_{1}^{n_{0}}(p'),g_{1}^{n_{0}+1}(p'))$. Observe that, for any $i \geq 0$, $I_{i} \subset [g_{1}^{n_{0}+i}(p'),g_{1}^{n_{0}+i+1}(p'))$. Hence the intervals $I_{i}$, for $i \geq 0$ are pairwise disjoint and
 $$ \begin{array}{rcccc}
 \log(|\frac{(g_{1}^{n})'(x)}{(g_{1}^{n})'(g_{2}^{k}(x))}|) &\leq & \displaystyle k_{1} \sum_{i=0}^{n-1} |I_{i}| & \leq & k_{1} D.
 \end{array}$$ 
Then, by $(*)$, for any $k>0$, $ \sup_{x \in [p',p] \cap K} |( g_{2}^k)'(x) | \leq M$, where $M=e^{k_{1}D}$.

Now let us prove that the diffeomorphism $g_{2}$ fixes the points in $[p',p] \cap K$. Suppose for a contradiction that there exists 
 a point $x_{0}$ in $[p',p] \cap K$ which is not fixed under $g_{2}$. Let $I$ be the connected component of $[p',p] - \mathrm{Fix}(g_{2})$ which contains $x_{0}$. Take a point $y$ in $I-K$ and let $(y_{-},y_{+})$ be the connected component of $I-K$ which contains $y$. Then, for any $k>0$,
$$\frac{(\tilde{g}_{2}^{k})'(y)}{(\tilde{g}_{2}^{k})'(y_{-})} \leq e^{k_{2}D}.$$
Hence $(\tilde{g}_{2}^{k})'(y) \leq M e^{k_{2}D}$. This implies that, for any $k \geq 0$, $ \sup_{x \in I} |( \tilde{g}_{2}^k)'(x) | \leq Me^{k_{2}D}$. Hence, using the mean value theorem, we see that the diffeomorphism $\tilde{g}_{2}$ has to fix the points of $I$, a contradiction.
\end{proof}

Propostion \ref{outsidemin} is a consequence of the following weaker Proposition.

\begin{proposition} \label{outsideminweak}
Suppose that the point $p$ is isolated on the left in $\mathrm{Fix}(G_{1})$. Then the group $\mathcal{G}_{1,p}$ is metabelian, \emph{i.e.} the group $\mathcal{G}'_{1}$ is abelian.
\end{proposition}

\begin{proof}[Proof of Proposition \ref{outsideminweak}]
By Corollary \ref{nofixedpoint}, there exists an element $f$ of $G_{1}$ and $\alpha'_{0}>0$ such that the diffeomorphism $f$ has no fixed point on $(p-\alpha'_{0},p)$ and is increasing on the interval $[p-\alpha'_{0},p]$. Taking $f^{-1}$ instead of $f$ if necessary, we can suppose that, for any point $x$ in $(p-\alpha'_{0},p) \cap K$, $f(x) >x$.

Fix two nontrivial elements $\xi_{1}$ and $\xi_{2}$ of the group $\mathcal{G}'_{1}$. We want to prove that $\xi_{1} \xi_{2}= \xi_{2} \xi_{1}$. Denote by $\mathcal{G}_{2}$ the subgroup of $\mathcal{G}_{1,p}$ generated by the elements $\xi_{1}$ and $\xi_{2}$. The advantage of considering this subgroup instead of $\mathcal{G}'_{1}$ is that the group $\mathcal{G}_{2}$ is finitely generated whereas the group $\mathcal{G}'_{1}$ might not be finitely generated.

Fix respective representative $g_{1}$ and $g_{2}$ of $\xi_{1}$ and $\xi_{2}$ in the group $G'_{1}$ and let $K_{2} \subset K$ the set of fixed points of the group $G_{2}$ generated by $g_{1}$ and $g_{2}$. Observe that, by Proposition \ref{firststep}, the set $\mathrm{Fix}(G'_{1}) \subset K_{2}$ accumulates on the left of the point $p$. Take a point $p-\alpha'_{0}<p' < p$ such that
\begin{enumerate}
\item The diffeomorphisms $g_{1}$ and $g_{2}$ are increasing on $[p',p]$.
\item The point $p'$ belongs to the set $\mathrm{Fix}(G'_{1})$.
\end{enumerate} 
An easy induction on wordlength proves that any element of $G_{2}$ is increasing on $[p',p]$.

We will distinguish two cases depending on the existence of a fixed point which is outside $K_{2}$ for an element of $G_{2}$. As the elements $\xi_{1}$ and $\xi_{2}$ are supposed to be nontrivial, the set $(K \setminus K_{2}) \cap [p',p]$ is nonempty.

\emph{First case:} Suppose that, for any element $g$ in the group $G_{2}$ and any connected component $(p_{1},p_{2})$ of $[p',p] \setminus K_{2}$ which meets $K$, either
$$ \mathrm{Fix}(g) \cap (p_{1},p_{2}) = \emptyset$$
or $g$ is equal to the identity on $(p_{1},p_{2}) \cap K$.
 Take a connected component $(p_{1},p_{2})$ of $[p',p] \setminus K_{2}$ which meets $K$. Denote by $G_{2|(p_{1},p_{2})}$ be the group of restrictions to $(p_{1},p_{2}) \cap K$ of elements of the group $G_{2}$. We now use the following lemma which is once again a straightforward consequence of Lemma \ref{Archimedean}.

\begin{lemma} \label{Holderbis}
Let $(p'_1,p'_2)$ be an open interval of $\mathbb{R}$ which meets $K$ and whose endpoints belong to the Cantor set $K$. Let $G$ be a subgroup of $\mathfrak{diff}^1(K)$. Suppose that
\begin{enumerate}
\item Any diffeomorphism in the group $G$ preserves $(p'_1,p'_2) \cap K$ and is increasing on the interval $(p'_1,p'_2)$.
\item Any element of $G$ which has a fixed point in $(p'_1,p'_2) \cap K$ is equal to the identity on $(p'_1,p'_2) \cap K$.
\end{enumerate}
Then the group $G_{|(p'_1,p'_2)}$ of restrictions to $(p'_1,p'_2) \cap K$ of elements of $G$ is abelian.
\end{lemma}

\begin{proof}
As any nontrivial element of the group $G$ is increasing on $(p'_1,p'_2)$ and has no fixed point on $(p'_1,p'_2) \cap K$, then, for any nontrivial element $g$ of the group $G_{|(p'_1,p'_2)}$, either
$$ \forall x \in (p'_1,p'_2) \cap K, \ g(x)>x$$
or
$$ \forall x \in (p'_1,p'_2) \cap K, \ g(x)<x.$$
Let $g$ and $h$ be two elements of the group $G$. Hence, if there exists a point $x_{0}$ in $(p'_1,p'_2) \cap K$ such that $g(x_{0}) <h(x_{0})$ then, for any point $x$ of $(p'_1,p'_2) \cap K$, $g(x) < h(x)$: otherwise the diffeomorphism $g^{-1}h$ would be nontrivial and would have a fixed point in $(p'_1,p'_2) \cap K$, which is not possible.
We define then an order $\preceq$ on $G_{|(p'_1,p'_2)}$ by setting $g \preceq h$ if and only if there exists a point $x_{0}$ of $(p'_1,p'_2) \cap K$ such that $g(x_{0}) \leq h(x_{0})$. The above remark proves that this defines a total order on the group $G_{|(p'_1,p'_2)}$. With a proof similar to the proof of Lemma \ref{Holder} (and even easier!), we can show that this defines a biinvariant Archimedean order on the group $G_{|(p'_1,p'_2)}$. By Lemma \ref{Archimedean}, the group $G_{|(p'_1,p'_2)}$ is abelian.
\end{proof}
 
By Lemma \ref{Holderbis}, the group $G_{2|(p_{1},p_{2})}$ is abelian. As this lemma is true for any such connected component $(p_{1},p_{2})$, we deduce that 
 $$ \forall x \in [p',p] \cap K, g_{1}g_{2}(x)=g_{2}g_{1}(x).$$
Hence $\xi_{1} \xi_{2}= \xi_{2} \xi_{1}$.

\emph{Second case:} There exists an element $g$ in the group $G_{2}$ and a point $p_{0}$ in $(K \setminus K_{2}) \cap [p',p]$ with the following properties.
\begin{enumerate}
\item $g(p_{0})=p_{0}$.
\item If we denote by $(p_{1},p_{2})$ the connected component of $[p',p] \setminus K_{2}$ which contains the point $p_{0}$,
$$ g_{|(p_{1},p_{2}) \cap K} \neq Id_{(p_{1},p_{2}) \cap K}.$$
\end{enumerate}
 Without loss of generality, we can suppose that the point $p_{0}$ is one of the endpoints of a connected component of $[p',p] \setminus \mathrm{Fix}(g)$ which meets the Cantor set $K$. We will find a contradiction, namely we will construct a free subsemigroup on two generators of the group $G_{1}$. Hence the first case always holds.

As the point $p_{0}$ does not belong to the set $K_{2}= \mathrm{Fix}(G_{2})$, there exists a diffeomorphism $h$ in $G_{2}$ such that $h(p_{0})>p_{0}$. 

Let $(a,b)$ be the connected component of $\mathbb{R} \setminus \mathrm{Fix}(G'_{1})$ which contains the points $p_{0}$ and $h(p_{0})$. Recall we fixed an element $f$ at the beginning of the proof of the proposition. Take $N'>0$ sufficiently large so that $f^{N'}(a)>b$ and let $f_{1}=f^{N'}$. Observe that the sets $f_{1}^{n}([a,b] \cap K)$, for $n \geq 0$, are pairwise disjoint. Moreover, as the set $\mathrm{Fix}(G'_{1})$ is invariant under the action of the group $G_{1}$, for any $n \geq 0$, the points $f_{1}^{n}(a)$ and $f_{1}^{n}(b)$ are fixed under the elements of $G_{2} < G'_{1}$. 

The proof uses the following lemma which relies on distortion estimates.  

\begin{claim} \label{distortion}
There exists an integer $N \geq 0$ such that, for any $n \geq N$,
$$h(f_{1}^{n}(p_{0}))=f_{1}^{n}(p_{0}).$$
\end{claim}

Before proving this claim, let us see how we can obtain a contradiction from this claim. More precisely, we want to prove that the semigroup generated by $f_{1}$ and $h$ is free, a contradiction. We will use the following lemma to do so.

\begin{lemma} \label{semifree}
Let $p'<p$ be points of $K$. Suppose that the point $p$ is accumulated on the left by points of $K$. Let $f$ and $h$ be diffeomorphisms in $\mathfrak{diff}^{1}(K)$ such that
\begin{enumerate}
\item The diffeomorphism $f$ is increasing on $[p',p]$ and 
$$ \left\{
\begin{array}{l}
f(p)=p \\
\forall x \in [p',p) \cap K, \ f(x)>x
\end{array}
\right.
.$$
\item There exists a point $p_{*} \in (p',p) \cap K$ such that
$$ \forall n \geq 0, h(f^{n}(p_{*}))=f^{n}(p_{*}).$$
\item There exists a point $p_{0}$ in $[p',p_{*}) \cap K$ and an integer $N > 0$ such that
$$\left\{ \begin{array}{l} 
h(p_{0}) \neq p_{0} \\
\forall n \geq N, h(f^{n}(p_{0}))=f^{n}(p_{0}).
\end{array}
\right.$$
\end{enumerate}
Then there exists $N'>0$ such that the semigroup generated by $f^{N'}$ and $f^{N'}h$ is free.
\end{lemma}

To obtain the wanted contradiction, use the above lemma with $f_{1}$, $h$ and $p_{*}=b$. The last hypothesis is satisfied thanks to Claim \ref{distortion}. To complete the proof of Proposition \ref{outsideminweak}, it suffices to prove those two lemmas.

\begin{proof}[Proof of Claim \ref{distortion}]
Recall that the point $p_{0}$ is the endpoint of a connected component of $\mathbb{R} \setminus \mathrm{Fix}(g)$  which meets $K$. Suppose that this connected component is of the form $(p_{0},p_{1})$ (the case where it is of the form $(p_{1},p_{0})$ is analogous). Observe that the interval $(f_{1}^{n}(p_{0}), f_{1}^{n}(p_{1}))$ is a connected component of $\mathbb{R} \setminus \mathrm{Fix}(f_{1}^{n}gf_{1}^{-n})$. Then, as the elements $h$ and $f_{1}^{n}gf_{1}^{-n}$, as well as the elements $h^{-1}$ and $f_{1}^{n}gf_{1}^{-n}$, are not crossed by Lemma \ref{crossed}, either $h(f_{1}^{n}(p_{0}))=f_{1}^{n}(p_{0})$ or $h(f_{1}^{n}(p_{0}))\geq f_{1}^{n}(p_{1})$ or $h(f_{1}^{n}(p_{1})) \leq f_{1}^{n}(p_{0})$: if none of those statements occur,
\begin{enumerate}
\item either $h(f_{1}^{n}(p_{0})) \in (f_{1}^{n}(p_{0}),f_{1}^{n}(p_{1}))$ and the elements $h$ and $f_{1}^{n}gf_{1}^{-n}$ are crossed,
\item or $h(f_{1}^{n}(p_{1})) \in  (f_{1}^{n}(p_{0}),f_{1}^{n}(p_{1}))$ and the elements $h$ and $f_{1}^{n}gf_{1}^{-n}$ are crossed,
 \item or $h(f_{1}^{n}(p_{0})) \leq f_{1}^{n}(p_{0})$ and $h(f_{1}^{n}(p_{1})) \geq f_{1}^{n}(p_{1})$. In this case, $h^{-1}(f_{1}^{n}(p_{0})) \geq f_{1}^{n}(p_{0})$ and $h^{-1}(f_{1}^{n}(p_{1})) \leq f_{1}^{n}(p_{1})$ and the elements $h^{-1}$ and $f_{1}^{n}gf_{1}^{-n}$ are crossed.
\end{enumerate}
Suppose for a contradiction that there exists a sequence of integer $n_{k} \rightarrow +\infty$ as $k \rightarrow +\infty$ such that, for any $k \geq 0$, $h(f_{1}^{n_{k}}(p_{0})) \geq f_{1}^{n_{k}}(p_{1})$. Then, for any $k \geq 0$,

$$ \frac{h(f_{1}^{n_{k}}(p_{0}))-h(f_{1}^{n_{k}}(a))}{f_{1}^{n_{k}}(p_{0})-f_{1}^{n_{k}}(a)}= \frac{h(f_{1}^{n_{k}}(p_{0}))-f_{1}^{n_{k}}(a)}{f_{1}^{n_{k}}(p_{0})-f_{1}^{n_{k}}(a)}$$
as the point $f_{1}^{n_{k}}(a)$ belongs to $\mathrm{Fix}(G'_{1})$ and the diffeomorphism $h$ belongs to $G'_{1}$.
Hence
$$ \begin{array}{rcl}
\frac{h(f_{1}^{n_{k}}(p_{0}))-h(f_{1}^{n_{k}}(a))}{f_{1}^{n_{k}}(p_{0})-f_{1}^{n_{k}}(a)} & \geq &  \frac{f_{1}^{n_{k}}(p_{1})-f_{1}^{n_{k}}(a)}{f_{1}^{n_{k}}(p_{0})-f_{1}^{n_{k}}(a)} \\
 & \geq & \frac{f_{1}^{n_{k}}(p_{1})-f_{1}^{n_{k}}(p_{0})}{f_{1}^{n_{k}}(p_{0})-f_{1}^{n_{k}}(a)}+1.
 \end{array}
$$
Denote by $\tilde{f}_{1}:[a,p] \rightarrow [a,p]$ a $C^{1+Lip}$-diffeomorphism such that $\tilde{f}_{1|[a,p] \cap K}=f_{1|[a,p] \cap K}$. By the mean value theorem, there exist points $c_{1}$ and $c_{2}$ of the interval $(a,b)$ such that
$$\left\{
\begin{array}{l}
\tilde{f}_{1}^{n_{k}}(p_{1})-\tilde{f}_{1}^{n_{k}}(p_{0})=(\tilde{f}_{1}^{n_{k}})'(c_{1})(p_{1}-p_{0}) \\
\tilde{f}_{1}^{n_{k}}(p_{0})-\tilde{f}_{1}^{n_{k}}(a)=(\tilde{f}_{1}^{n_{k}})'(c_{2})(p_{0}-a).
\end{array}
\right.$$

Moreover, if we denote by $K$ the Lipschitz constant of $\log(f')$ then
$$ \begin{array}{rcl} \left| \log((\tilde{f}_{1}^{n_{k}})'(c_{1}))- \log((\tilde{f}_{1}^{n_{k}})'(c_{2})) \right| & \leq & \displaystyle K \sum_{i=0}^{n_{k}-1} \left| \tilde{f}_{1}^{i}([a,b]) \right| \\
 & \leq & K \left| p-a \right|=M,
 \end{array}
$$
because the intervals $\tilde{f}_{1}^{i}([a,b])$ are pairwise disjoint and contained in the interval $(a,p)$. Hence
$$\frac{h(f_{1}^{n_{k}}(p_{0}))-h(f_{1}^{n_{k}}(a))}{f_{1}^{n_{k}}(p_{0})-f_{1}^{n_{k}}(a)}  \geq   e^{-M} \frac{p_{1}-p_{0}}{p_{0}-a} +1 >1.$$
However, recall that fixed points of $h$ accumulate to $p$ because $h \in G'_{1}$. Hence $h'(p)=1$, in contradiction with the above inequality and the continuity of $h'$.

In the case where there exists a sequence of integers $n_{k}$ which tends to $+\infty$ as $k \rightarrow + \infty$ such that, for any $k$, $h(f^{n_{k}}(p_{1})) \leq f^{n_{k}}(p_{0})$ (or $f^{n_{k}}(p_{1}) \leq h^{-1}(f^{n_{k}}(p_{0}))$), we find a similar contradiction by using $h^{-1}$ instead of $h$.
\end{proof}

\begin{proof}[Proof of Lemma \ref{semifree}]
Let $N'-1$ be the largest integer $n$ such that $h(f^{n}(p_{0})) \neq f^{n}(p_{0})$. Let $f_{1}=f^{N'}$. Hence, for any integer $n > 0$, the point $f_{1}^{n}(p_{0})$ is fixed under $h$ (but the point $p_{0}$ is not). We want to prove that the semigroup generated by $f_{1}$ and $g=f_{1}h$ is free. Let
$$ w_{1}= f_{1}^{n_{k}} g^{m_{k}} \ldots f_{1}^{n_{2}} g^{m_{2}} f_{1}^{n_{1}} g^{m_{1}},$$
with $n_{k} \geq 0$, $m_{1} \geq  0$ and  $n_{i}>0$, $m_{i}>0$ otherwise, and  
$$ w_{2}= f_{1}^{n'_{k'}} g^{m'_{k'}} \ldots f_{1}^{n'_{2}} g^{m'_{2}} f_{1}^{n'_{1}} g^{m'_{1}},$$
with $n'_{k'} \geq 0$, $m'_{1} \geq  0$ and  $n'_{i}>0$, $m'_{i}>0$ otherwise, be two distinct words on $f_{1}$ and $g$. We see each of those words as a diffeomorphism in $\mathfrak{diff}^{1}(K)$. Suppose for a contradiction that, as elements of the group $\mathfrak{diff}^{1}(K)$, $w_{1}=w_{2}$. Then, simplifying these words on the right and exchanging the roles of $w_{1}$ and $w_{2}$ if necessary, we can suppose that $m_{1}>0$ and $m'_{1}=0$. Now, let us look at the image of the point $p_{*}$ under those two diffeomorphisms. We have
$$ \left\{
\begin{array}{l}
w_{1}(p_{*})= f_{1}^{n_{1}+m_{1}+n_{2}+m_{2}+ \ldots +n_{k}+m_{k}}(p_{*}) \\
w_{2}(p_{*})= f_{1}^{n'_{1}+m'_{1}+n'_{2}+m'_{2}+ \ldots +n'_{k'}+m'_{k'}}(p_{*})
\end{array}
\right.
.$$

Hence $ \displaystyle \sum_{i=1}^{k} (n_{i}+m_{i})=\sum_{i=1}^{k'} (n'_{i}+m'_{i})$. We denote by $l$ the common value of these sums. Now, we will prove that $w_{1}(p_{0}) \neq w_{2}(p_{0})$, in contradiction with the equality $w_{1}=w_{2}$ as elements of the group $\mathfrak{diff}^{1}(K)$.

We write
$$ w_{1}=w_{3}h$$
where 
$$w_{3}=f_{1}^{n_{k}} g^{m_{k}} \ldots f_{1}^{n_{2}} g^{m_{2}} f_{1}^{n_{1}} g^{m_{1}-1}f_{1}.$$
Observe that $w_{2}(p_{0})=w_{3}(p_{0})=f_{1}^{l}(p_{0})$. As we supposed that $w_{1}=w_{2}$ as elements of $\mathfrak{diff}^{1}(K)$, we have $w_{3}h(p_{0})=w_{1}(p_{0})=w_{2}(p_{0})$ hence $h(p_{0})=p_{0}$, a contradiction.
\end{proof}
\end{proof}

Now, let us deduce Proposition \ref{outsidemin} from Proposition \ref{outsideminweak}.

\begin{proof}[Proof of Proposition \ref{outsidemin}]
 By Proposition \ref{outsideminweak}, the group $\mathcal{G}_{1,p}$ is metabelian, meaning that its derived subgroup is abelian. A theorem by Rosenblatt (see \cite{Ros}) states that any metabelian group without free subsemigroups on two generators is nilpotent. Hence the group $\mathcal{G}_{1,p}$ is nilpotent.

Suppose for a contradiction that the group $\mathcal{G}_{1,p}$ is not abelian and take a nontrivial element $\xi$ in the center of the group $\mathcal{G}_{1,p}$ which belongs to the derived subgroup $\mathcal{G}'_{1}$. 

Recall that, by Corollary \ref{nofixedpoint}, the group $\mathcal{G}_{1,p}$ contains an element whose representative has no fixed point in a left-neighbouhood of the point $p$. Hence, by Lemma \ref{Kopell}, the element $\xi$ has a representative $g_{\xi}$ in $G_{1}$ with no fixed point on a left-neighbourhood of the point $p$.

We want to apply Lemma \ref{Holder} to prove that the group $\mathcal{G}_{1,p}$ is abelian and finish the proof of Proposition \ref{outsidemin}. In the rest of this proof, we make sure that the hypothesis of Lemma \ref{Holder} are satisfied.

Take $\alpha_{0}>0$ such that the diffeomorphisms $g_{\xi}$ and $g_{\xi}^{-1}$ are increasing on $[p-\alpha_{0},p]$ and have no fixed point in $[p-\alpha_{0},p) \cap K$. Taking $\xi^{-1}$ instead of $\xi$ if necessary, we can further suppose that, for any point $x$ of $[p-\alpha_{0},p] \cap K$,
$$ g_{\xi}(x) >x.$$

Suppose for a contradiction that there exist a real number $\alpha \in (0,\alpha_{0})$ and a diffeomorphism $g \in G_{1}$ with the following properties.
\begin{enumerate}
\item The diffeomorphism $g$ is increasing on $[p-\alpha,p]$.
\item The diffeomorphism $g$ has a fixed point $p_{1}$ in $[p-\alpha,p) \cap K$.
\item There exists a point $x_{0}$ in $[p_{1},p] \cap K$ such that $g(x_{0}) \neq x_{0}$.
\end{enumerate}
As $g([p_{1},p])=[p_{1},p]$ and $g_{\xi}([p_{1},p]) \subset [p_{1},p]$, we have 
$$g^{-1}g_{\xi}g([p_{1},p]) \subset [p_{1},p] \subset (p-\alpha_{0},p)$$
and the diffeomorphism $h=[g_{\xi},g]=g_{\xi}^{-1}g^{-1}g_{\xi}g$ is increasing on $[p_{1},p]$. Take $\alpha' >0$ small enough so that $p-\alpha_{0} <p_{1}-\alpha'$ and the diffeomorphism $h$ is increasing on $(p_{1}-\alpha',p]$.

\begin{lemma} \label{commutator}
For any point $x$ in $(p_{1}-\alpha',p] \cap K$, $h(x) = x$
\end{lemma}

\begin{proof}
Suppose that there exists a point $p_{0}$ in $(p_{1}-\alpha',p]\cap K$ such that $h(p_{0}) \neq p_{0}$. As the element $\xi$ of $\mathcal{G}_{1,p}$ lies in the center of $\mathcal{G}_{1,p}$, the diffeomorphism $h$ pointwise fixes a left-neighbourhood of $p$ in $K$. Let
$$ p_{*}= \inf \left\{ x \in [p_{0},p], \forall y \in [x,p] \cap K, h(y)=y \right\}.$$
Then apply Lemma \ref{semifree} to $h$ and $f=g_{\xi}$ to find a free subsemigroup of $G_{1}$ on two generators, a contradiction.
\end{proof}

By Lemma \ref{commutator}, for any point $x$ in $(p_{1}-\alpha',p]\cap K$,
$$ gg_{\xi}(x)=g_{\xi}g(x).$$
Hence, by Lemma \ref{Kopell}, the diffeomorphism $g$ has no fixed point in $(p_{1}-\alpha',p]$, a contradiction.

Therefore, we can apply Lemma \ref{Holder} and the group $\mathcal{G}_{1,p}$ is abelian.
\end{proof}

\begin{proof}[Proof of Proposition \ref{outsideminacc}]
Fix a finite generating set $S$ of $G_{1}$. As the point $p$ is accumulated on the left by points of $\mathrm{Fix}(G_{1})$, for any diffeomorphism $s$ in $S$, there exists a point $p_{s}<p$ in $\mathrm{Fix}(G_{1})$ such that $s$ is increasing on the interval $[p_{s},p]$. Let
$$ p' = \max \left\{ p_{s} \mid s \in S \right\}.$$
Then any element of $S$ is increasing on $L_{p}=[p',p]$ and preserves $L_{p} \cap K$. Hence any element of $G_{1}= \langle S \rangle$ is increasing on $L_{p}$ and preserves $L_{p} \cap K$.

Now, let us prove by contradiction that the group $G'_{1}$ acts trivially on $L_{p} \cap K$. Suppose there exists a point $x_{0} \in (p',p) \cap K$ which is displaced by some element of $G'_{1}$. Let $(p_{1},p_{2})$ be the connected component of $(p',p) \setminus \mathrm{Fix}(G_{1})$ which contains the point $x_{0}$. Then the points $p_{1}$ and $p_{2}$ belong to $\mathrm{Fix}(G_{1})$. Moreover, the point $p_{1}$ is accumulated on the right by points of $K$: otherwise, if $I$ is connected component of $\mathbb{R} \setminus K$ whose left-end is the point $p_{1}$, then its right-end is a fixed point of $G_{1}$, in contradiction with $(p_{1},p_{2}) \cap \mathrm{Fix}(G_{1}) = \emptyset$. Likewise, the point $p_{2}$ is accumulated on the left by points of $K$. By Corollary \ref{nofixedpoint}, there exists a point $p'_{2}<p_{2}$ of $K$ and an element $f$ of $G_{1}$ such that 
$$ \forall x \in [p'_{2},p_{2}) \cap K, \ f(x) >x.$$
We then need the following lemma.

\begin{lemma} \label{nbhdacc}
$$[p'_{2},p_{2}] \cap K \subset \mathrm{Fix}(G'_{1}).$$
\end{lemma}

Of course, we can likewise prove that there exists a point $p'_{1}>p_{1}$ of $K$ such that $[p_{1},p'_{1}] \cap K \subset \mathrm{Fix}(G'_{1})$.

Before proving this lemma, let us see why it gives us the wanted contradiction. The set
$$ A = \overline{ \left\{ x \in [p_{1},p_{2}] \cap K \mid \exists g_{1} \in G'_{1}, \ g_{1}(x) \neq x \right\} }$$
is a closed $G_{1}$-invariant set as the group $G'_{1}$ is a normal subgroup of $G_{1}$ and any element of $G_{1}$ preserves $[p_{1},p_{2}] \cap K$. Moreover, by Lemma \ref{nbhdacc}, this set $A$ is contained in $[p'_{1},p'_{2}]$. Also, this set contains a minimal invariant set for the action of $G_{1}$ on $K$, hence a fixed point for $G_{1}$ by Proposition \ref{minpoint}. Hence
$$ \emptyset \neq A \cap \mathrm{Fix}(G_{1}) \subset [p'_{1},p'_{2}] \cap \mathrm{Fix}(G_{1}) \subset (p_{1},p_{2}) \cap \mathrm{Fix}(G_{1}) = \emptyset,$$
a contradiction.
\end{proof}

\begin{proof}[Proof of Lemma \ref{nbhdacc}]
Suppose for a contradiction that there exists a point $x_{0} \in [p'_{2},p_{2}] \cap K$ and an element $h$ in $G'_{1}$ such that $h(x_{0}) \neq x_{0}$. Then let
$$ p_{max}= \sup \left\{ x \in [p'_{2},p_{2}] \cap K, h(x) \neq x \right\}.$$
By Proposition \ref{outsidemin}, $p_{max} < p_{2}$. Observe that this point $p_{max}$ is fixed under the diffeomorphism $h$. 

Finally, take a point $p_{0} < p_{max}$ of $K$ such that $h(p_{0}) \neq p_{0}$ and $f(p_{0}) > p_{max}$. Taking $h^{-1}$ instead of $h$ if necessary, we can suppose that $h(p_{0}) >p_{0}$. Then, by Lemma \ref{semifree}, the group generated by $h$ and $f$ contains a free semigroup on two generators, a contradiction.
\end{proof}

\subsection{End of the proof of Theorem \ref{nosubsemi}} \label{44}

Now, we finish the proof of Theorem \ref{nosubsemi}, namely we prove the following proposition.

\begin{proposition} \label{periodic}
Let $G_1$ be a subgroup of $G$ which satisfies Proposition \ref{minpoint}. Then
its derived subgroup $G'_{1}$ is trivial.
\end{proposition}

\begin{proof}[Proof of Proposition \ref{periodic}]
For any point $x$ in $\mathrm{Fix}(G_{1})$, we want to define a left neighbourhood $L_{x}$ and a right-neighbourhood $R_{x}$ of the point $x$ in $K$ which will be useful for the proof. To define those, we have to distinguish cases.
\begin{enumerate}
\item If the point $x$ is accumulated on the left (respectively the right) by points of $\mathrm{Fix}(G_{1})$, take a left-neighbourhood $L_{x}$ (resp. a right-neighbourhood $R_{x}$) of $x$ such that the set $L_{x}$ (resp. $R_{x}$) is pointwise fixed under the elements of $G'_{1}$ (such a neighbourhood exists by Proposition \ref{outsideminacc}).
\item If the point $x$ is accumulated on the left (resp. the right) by points of $K$ but isolated on the left (resp. the right) in $\mathrm{Fix}(G_{1})$ then take a left-neighbourhood $L_{x}$ (resp. a right-neighbourhood $R_{x}$) of $x$ such that there exists a diffeomorphism $f$ in $G_{1}$ such that
\begin{enumerate}
\item The diffeomorphism $f$ is increasing on $L_{x}$ (resp. $R_{x}$).
\item For any point $y$ in $L_{x} \setminus \left\{ x \right\}$, $f(y)>y$ (resp. for any point $y$ in $R_{x} \setminus \left\{ x \right\}$, $f(y)<y$). 
\end{enumerate}
Such a diffeomorphism $f$ exists by Corollary \ref{nofixedpoint}.
\item If the point $x$ is isolated on the left (resp. on the right) in $K$, take $L_{x}= \left\{ x \right\}$ (resp. $R_{x}= \left\{ x \right\}$).
\end{enumerate}
Let
$$ U= \bigcup_{x \in \mathrm{Fix}(G_{1})} (L_{x} \cup R_{x})$$
and choose the neighbourhoods $L_{x}$ and $R_{x}$ in such a way that the set $U$ is open in $K$.
Suppose for a contradiction that the group $G'_{1}$ contains anontrivial element $h$. Let $A$ be a subset of $K$ consisting of points $x$ of $K$ which are displaced under some element of the group $G'_{1}$, \emph{i.e.} there exists an element $h$ of $G'_{1}$ such that $h(x) \neq x$. Of course, this set is disjoint from the set $\mathrm{Fix}(G_{1})$. We can say even more by the following claim.

\begin{claim} \label{nbhdfixed}
$A \cap U = \emptyset$.
\end{claim}

Before proving the claim, let us see how we can finish the proof of Proposition \ref{periodic}. As the set $U$ is open, $\overline{A} \cap U = \emptyset$. 

As the group $G'_{1}$ is a normal subgroup of the group $G_{1}$, the set $\overline{A}$ is a closed $G_{1}$-invariant subset. Hence there exists a minimal set $M \subset \overline{A}$ for the action of $G_{1}$ on $\overline{A}$. By Proposition \ref{minpoint}, $M \subset \mathrm{Fix}(G_{1}) \subset U$, a contradiction with Claim \ref{nbhdfixed}.
\end{proof}

\begin{proof}[Proof of Claim \ref{nbhdfixed}]
Suppose for a contradiction that $A \cap U \neq \emptyset$ and take a point $p_{0}$ in the intersection $A \cap U$. By definition of the set $A$, there exists an element $h$ in $G'_{1}$ such that $h(p_{0}) \neq p_{0}$.

By definition of $U$, there exists a point $p$ in $\mathrm{Fix}(G_{1})$ such that either $p_{0} \in L_{p}$ or $p_{0} \in R_{p}$. Suppose for instance that the point $p_{0}$ belongs to the left-neighbourhood $L_{p}$ of $p$. Necessarily, the point $p$ is isolated on the left in $\mathrm{Fix}(G_{1})$: otherwise, the diffeomorphisms in $G'_{1}$ pointwise fix $L_{p}$. Moreover, by construction of $L_{p}$ there exists an element $f$ in $G_{1}$ such that
\begin{enumerate}
\item For any point $y$ in $L_{p} \setminus \left\{ p \right\}$, $f(y)>y$.
\item The diffeomorphism $f$ is increasing on $L_{p}$.
\end{enumerate}

Finally, we use Lemma \ref{semifree} (recall that the diffeomorphism $h$ pointwise fixes a neighbourhood of $p$ to find the point $p_{*}$). By this lemma, the group $G_{1}$ contains a free subsemigroup on two generators, a contradiction.
\end{proof}


\begin{thebibliography}{1}

\bibitem{BBG} C. Bleak, H. Bowman, A. Gordon Lynch, G. Graham, J. Hughes, F. Matucci, E. Sapir, \emph{Centralizers in the R. Thompson group Vn}, Groups Geom. Dyn. 7 (2013), no. 4, 821-865. 

\bibitem{BDJ} C. Bleak, C. Donoven, J. Jonušas,\emph{Some isomorphism results for Thompson-like groups Vn(G)}, Israel J. Math. 222 (2017), no. 1, 1-19. 

\bibitem{Bur} W. Burnside, \emph{On an unsettled question in the theory of discontinuous groups}, Quart. J. Math., vol. 33, 1902, 230-238.

\bibitem{FF} B. Farb, J. Franks, \emph{Groups of homeomorphisms of one-manifolds. III. Nilpotent subgroups.} Ergodic Theory Dynam. Systems 23 (2003), no. 5, 1467-1484.

\bibitem{FN} L. Funar, Y. Neretin, \emph{Diffeomorphism groups of tame Cantor sets and Thompson-like groups}, Compos. Math. 154 (2018), no. 5, 1066-1110.

\bibitem{GoS} E. Golod, I. R. Shafarevich, \emph{On the class field tower}, Izv. Akad. Nauk SSSR Ser. Mat. 28, 1964, 261-272.

\bibitem{dlH} P. de la Harpe, \emph{Topics in geometric group theory}, Chicago Lectures in Mathematics. University of Chicago Press, Chicago, IL, 2000.

\bibitem{Hir} M. W. Hirsch, \emph{Differential topology}, corrected reprint of the 1976 original, Graduate Texts in Mathematics, 33, Springer-Verlag, New York, 1994.

\bibitem{Kop} N. Kopell, \emph{Commuting diffeomorphisms}, 1970, Global Analysis (Proc. Sympos. Pure Math., Vol. XIV, Berkeley, Calif., 1968), 165-184, Amer. Math. Soc., Providence, R.I.

\bibitem{Nav} A. Navas, \emph{Groups of circle diffeomorphisms}, translation of the 2007 Spanish edition, Chicago Lectures in Mathematics, University of Chicago Press, Chicago, IL, 2011. 

\bibitem{Nav2} A. Navas, \emph{Growth of groups and diffeomorphisms of the interval}, Geom. and Functional analysis 18(2008), 988-1028.

 
\bibitem{Ros} J. M. Rosenblatt, \emph{Invariant measures and growth conditions}, Trans. Amer. Math. Soc. 193 (1974), 33-53.

\bibitem{Rov} C.E. Rover, \emph{Constructing finitely presented simple groups that contain Grigorchuk groups}, J. Algebra 220 (1999), no. 1, 284-313.

\bibitem{Sch} Schachermayer, W., \emph{Addendum: Une modification standard de la d\'emonstration non standard de Reeb et Schweitzer  Un th\'eor\`eme de Thurston \'etabli au moyen de l'analyse non standard}, Differential topology, foliations and Gelfand-Fuks cohomology (Proc. Sympos., Pontificia Univ. Cat\'olica, Rio de Janeiro, 1976), pp. 139-140. Lecture Notes in Math., Vol. 652, Springer, Berlin, 1978.

\bibitem{Thu}  Thurston, William P., \emph{A generalization of the Reeb stability theorem}, Topology 13 (1974), 347-352.

\end{thebibliography}
\end{document}